\documentclass[leqno,11pt]{amsart}

\usepackage[left=2.8cm, right=2.8cm,top=2.6cm,bottom=2.6cm]{geometry}

\usepackage[x11names,dvispnames]{xcolor}
\definecolor{frenchblue}{rgb}{0.0, 0.45, 0.73}
\definecolor{electricyellow}{rgb}{1.0, 1.0, 0.0}
 \definecolor{davy\'sgrey}{rgb}{0.33, 0.33, 0.33}
 
	\definecolor{floralwhite}{rgb}{1.0, 0.98, 0.94}
	
	\definecolor{honeydew}{rgb}{0.94, 1.0, 0.94}
\usepackage{times}
\usepackage{tensor}
\usepackage[bbgreekl]{mathbbol}
\usepackage[all]{xy}
\usepackage{tikz}
\usepackage{tikz-cd}
\usepackage{pgfplots}
\usepgfplotslibrary{fillbetween}
\newsavebox{\pullback}
\sbox\pullback{%
\begin{tikzpicture}%
\draw (0,0) -- (1ex,0ex);%
\draw (1ex,0ex) -- (1ex,1ex);%
\end{tikzpicture}}
\usepackage{subcaption}

\usetikzlibrary{arrows}

\usepackage{amsmath, amssymb, amsfonts, latexsym, mdwlist, amsthm}

\usepackage{mathrsfs}
\usepackage{graphicx}
\usepackage{wrapfig}
\usepackage{longtable}
\usepackage{enumerate}
\usepackage{mathtools} 
\usepackage{bbm}
\usepackage{tkz-euclide}
\tikzset{
  dotted/.style={pattern=dots,pattern color=#1},
  dotted/.default=black
}
\tikzset{
  fdotted/.style={pattern=crosshatch dots,pattern color=#1},
  fdotted/.default=black
}
\tikzset{
  scopedlines/.style={pattern=north east lines,pattern color=#1},
  scopedlines/.default=black
}
\tikzset{
  hrlines/.style={pattern=horizontal lines,pattern color=#1},
  hrlines/.default=black
}
\newcommand*{\DashedArrow}[1][]{\mathbin{\tikz [baseline=-0.25ex,-latex, dashed,#1] \draw [#1] (0pt,0.5ex) -- (1.3em,0.5ex);}}

\makeindex

\usepackage[colorlinks=true,citecolor=blue, urlcolor=blue, linkcolor=blue]{hyperref}

\usepackage[maxbibnames=99,style=numeric, backend=biber, backref, backrefstyle=none, maxalphanames=6]{biblatex}
\def\P{\ensuremath{\mathbf{P}}}

\def\R{\ensuremath{\mathbf{R}}}
\def\Z{\ensuremath{\mathbf{Z}}}

\def\codim{\mathop{\mathrm{codim}}\nolimits}
\def\Cone{\mathop{\mathrm{cone}}\nolimits}

\def\Nef{\mathop{\mathrm{Nef}}\nolimits}

\def\Pic{\mathop{\mathrm{Pic}}\nolimits}

\def\Stab{\mathop{\mathrm{Stab}}\nolimits}

\DeclareMathSymbol\bbDelta  \mathord{bbold}{"01}

\renewcommand{\qedsymbol}{{$\Box$}}

\def\wGL2{\ensuremath{\widetilde{\mathrm{GL}_2^+(\R)}}}

\newenvironment{proof1}{\vspace{0.2cm}\paragraph{\bf\textit{Proof  of Theorem \ref{Thm:baselociHK}, items \textit{(1)}, \textit{(2)}, \textit{(3)}:}}}{\hfill \qedsymbol \medskip}

\newenvironment{proof3}{\vspace{0.2cm}\paragraph{\bf\textit{Proof  of Theorem \ref{Thm:duality}}}}{\hfill \qedsymbol \medskip}

\newenvironment{proof2}{\vspace{0.2cm}\paragraph{\bf\textit{Proof  of Theorem \ref{Thm:baselociHK}, item \textit{(4)}:}}}{\hfill \qedsymbol \medskip}

\newtheorem{maintheorem}{Theorem}

\newtheorem{corollary}{Corollary}

\newtheorem{Thm}{Theorem}[section]
\newtheorem{Prop}[Thm]{Proposition}
\newtheorem*{ThmA*}{Theorem A}
\newtheorem*{CorB*}{Corollary B}
\newtheorem*{ThmB*}{Theorem C}

\newtheorem{Lem}[Thm]{Lemma}

\newtheorem{Cor}[Thm]{Corollary}

\newtheorem*{Quesone*}{Question 1}
\newtheorem*{Questwo*}{Question 2}

\theoremstyle{definition}
\newtheorem{Def}[Thm]{Definition}
\newtheorem*{Def*}{Definition}
\newtheorem*{Rem*}{Remark}
\newtheorem{Rem}[Thm]{Remark}

\newtheorem{Ex}[Thm]{Example}

\makeatletter
\newtheoremstyle{italicsname}
 {3pt}
 {3pt}
 {\itshape}
 {}
 {\itshape}
 {.}
 {.5em}
 {\thmname{#1}\thmnumber{\@ifnotempty{#1}{ }#2}%
 \thmnote{ {\the\thm@notefont(#3)}}}
\makeatother
\theoremstyle{italicsname}

\numberwithin{equation}{section}

\makeatletter
\newcommand{\mylabel}[2]{#2\def\@currentlabel{#2}\label{#1}}
\makeatother

\title[Asymptotic base loci on HK manifolds]{Asymptotic base loci on hyper-Kähler manifolds}

\author[Francesco Antonio Denisi, Ángel David Ríos Ortiz]{Francesco Antonio Denisi,  Ángel David Ríos Ortiz}

\address{Université de Lorraine, Institut Elie Cartan de Lorraine, F-54506 Vandœuvre-lès-Nancy Cedex, France}
\curraddr{Université Paris Cité, Bâtiment Sophie Germain, 8 Place Aurélie Nemours, F-75205 Paris, France}
\email{denisi@imj-prg.fr}

\address{Max-Planck-Institut f\"{u}r Mathematik in den Naturwissenschaften, Inselstrasse 22, 04103 Leipzig, DE}
\curraddr{Universit\'e Paris-Saclay, CNRS, Laboratoire de Math\'ematiques d'Orsay, B\^at.\ 307, 91405 Orsay, France}
\email{angel-david.rios-ortiz@universite-paris-saclay.fr}

\usepackage{pgfplots}

\pgfplotsset{
compat=newest,
every axis plot/.append style={no marks,thick},
every axis/.style={
  axis lines=middle,
  width=7cm,
  height=3cm,
  }
}

\begin{document}

\begin{abstract} \noindent
Given a projective hyper-Kähler manifold $X$, we study the asymptotic base loci of big divisors on $X$. We provide a numerical characterization of these loci and study how they vary while moving a big divisor class in the big cone, using the divisorial Zariski decomposition, and the Beauville-Bogomolov-Fujiki form. We determine the dual of the cones of $k$-ample divisors $\mathrm{Amp}_k(X)$, for any $1\leq k \leq \mathrm{dim}(X)$, answering affirmatively (in the case of projective hyper-Kähler manifolds) a question asked by Sam Payne. We provide a decomposition for the effective cone $\mathrm{Eff}(X)$ into chambers of Mori-type, analogous to that for Mori dream spaces into Mori chambers. To conclude, we illustrate our results with several examples.
\end{abstract}

\maketitle
\section*{\bf \scshape Introduction}
A compact hyper-Kähler manifold (briefly HK manifold) is a simply connected compact Kähler manifold $X$ such that $H^0(X,\Omega^2_X)\cong \mathbf{C} \sigma$, where $\sigma$ is a holomorphic symplectic form. Compact HK manifolds form one of the three building blocks of compact Kähler manifolds with vanishing first (real) Chern class and only a few deformation classes of such objects are known to date, namely the $\mathrm{K3}^{[n]}$, $\mathrm{Kum}_n$ deformation classes, for any $n\in \mathbf{N}$ (cf.\ \cite{Beau}), and the deformation classes OG6, OG10, discovered by O'Grady (cf.\ \cite{OG1}, \cite{OG2} ). Thanks to the work \cite{Beau} of Beauville, there exists a quadratic form $q_X$ on $H^2(X,\mathbf{C})$, behaving like the intersection form on a surface. The quadratic form $q_X$ (whose explicit definition can be found, for example, in \cite[Definition 22.8]{GHJ2003}), is non-degenerate and integer-valued on $H^2(X,\mathbf{Z})$, and is known as Beauville-Bogomolov-Fujiki (BBF for short) form.
It follows that $H^2(X,\mathbf{Z}$) (resp.\ $\mathrm{Pic}(X)$) is endowed with a lattice structure. We refer the reader to \cite[Section 23]{GHJ2003}, or  \cite{Huybrechts1999}, for all the basic properties (with related proofs) of the BBF form.
\vspace{0.2cm}

Ample, nef, and semiample divisors play a central role in algebraic geometry. If $Y$ is any  normal, complex projective variety, to detect whether a Cartier divisor $D$ is ample, nef or semiample, one introduces the augmented base locus, the restricted base locus and the stable base locus of $D$, respectively, which are known as \textit{asymptotic base loci}. In symbols, we denote these by $\mathbf{B}_+(D),\mathbf{B}_-(D), \mathbf{B}(D)$, respectively. An integral Cartier divisor $D$ is ample (resp.\ nef, semiample) if and only if $\mathbf{B}_+(D)=\emptyset$ (resp.\ $\mathbf{B}_-(D)=\emptyset$, $\mathbf{B}(D)=\emptyset$). 
\vspace{0.2cm}

It has been established by Takayama (cf.\ \cite{Takayama}) that, whenever $Y$ is smooth and $K_Y\equiv 0$, given a big $\mathbf{Q}$-divisor $D$ on $Y$ (i.e.\ for any sufficiently divisible and large $k \in \mathbf{N}$ we have that $h^0(Y,\mathscr{O}_Y(kD)) $ grows like $ k^{\mathrm{dim}(Y)}$), all the irreducible components of $\mathbf{B}_+(D),\mathbf{B}_-(D), \mathbf{B}(D)$ are uniruled. 
The above result was generalized to $\mathbf{R}$-divisors, and to adjoint divisors by Boucksom, Broustet, Pacienza (cf.\ \cite[Theorem A]{BBP2013}). Moreover, in  \cite[Proof of Theorem A]{BBP2013}, one sees that for any irreducible component of the augmented (resp.\ restricted, stable) base locus of a big $\mathbf{R}$-divisor on $Y$, there exists a birational model $Y'$ of $Y$ on which the considered component can be contracted. Furthermore, \cite[Theorem A]{BBP2013} also implies that $\mathbf{B}(D)=\mathbf{B}_{-}(D)$ in this case, for any big divisor $D$.
\vspace{0.2cm}

In this paper, there are two major results. Before stating the first, we introduce some terminology.
\vspace{0.2cm}

We say that a big Cartier $\mathbf{R}$-divisor $D$ on a normal complex projective variety $Y$ is stable (resp.\ unstable) if $\mathbf{B}_+(D)=\mathbf{B}_-(D)$ (resp.\ $\mathbf{B}_-(D) \subsetneq \mathbf{B}_+(D)$). The stable divisor classes on $Y$ form a dense open subset of $\mathrm{Big}(Y)$, the big cone of $Y$ (i.e.\ the cone in $N^1(Y)_{\mathbf{R}}$ spanned by big integral divisor classes),  by \cite[Proposition 1.26]{Ein2}. Following \cite[Definition 2.8]{Bau}, we say that a real number $\lambda$ is a destabilizing number for a big, integral Cartier divisor $D$, with respect to an ample, integral Cartier divisor $A$, if $\mathbf{B}_+(D-\alpha A) \subsetneq \mathbf{B}_+(D-\beta A)$
for all rational numbers $\alpha,\beta$ with $\alpha < \lambda < \beta$. 
\vspace{0.2cm}

In general, given a big Cartier divisor $D$ on $Y$, the irreducible components of  $\mathbf{B}_+(D), \mathbf{B}_-(D)$ cannot be described using the intersections of $D$ with the curves on the variety (see Example \ref{Ex:notnumdet}). We recall that any pseudo-effective $\mathbf{R}$-divisor on a projective manifold admits a divisorial Zariski decomposition. 
\vspace{0.2cm}

In the case of a smooth projective surface $S$, the augmented and restricted base loci can be characterized by using the divisorial Zariski decomposition. In particular, Bauer, Küronya, and Szemberg prove in \cite{Bau} that if $D$ is any big divisor on $S$, $\mathbf{B}_+(D)$ is the union of the negative irreducible, reduced curves intersecting $P(D)$ trivially, while $\mathbf{B}_-(D)$ is the union of the curves supporting $N(D)$, where $P(D)$ (resp.\ $N(D)$) is the positive (resp.\ negative) part of the divisorial Zariski decomposition of $D$. Also, they show that all the destabilizing numbers are rational numbers and that $D$ is unstable if and only if there is a negative irreducible curve $C$ on $S$ such that $C \not \subset \mathrm{Supp}(N(D))$, and $P(D) \cdot C=0$. Here the point is that $P(D)$ encodes most of the positivity of $D$, and so it is possible to characterize $\mathbf{B}_+(D)$ using $P(D)$ and its intersections with the curves contained in $S$, while $\mathbf{B}_-(D)$ is determined by $N(D)$. In particular, on projective K3 surfaces, which are the projective hyper-Kähler manifolds of dimension 2, the above tells us that all the asymptotic base loci of big divisors are a finite union of smooth rational curves.
\vspace{0.2cm}

 We will say that a curve $C$ on a projective hyper-Kähler manifold $X$ intersects negatively (resp.\ trivially) a divisor $D$ and is contractible on some birational model of $X$, if there exist another projective hyper-Kähler manifold $X'$, a birational map $f\colon X \DashedArrow X'$, with $C \not \subset \mathrm{Ind}(f)$ (the indeterminacy locus of $f$), such that $f_*(D) \cdot f_*(C)<0$ (resp.\ $f_*(D) \cdot f_*(C)=0$) and $f_*(C)$ lies on an extremal face $F$ of $\overline{\mathrm{NE}(X)}$ which is contractible, where $f_*(C)$ is the strict transform of $C$ via $f$. We recall that a curve is extremal if its class in the Mori cone spans an extremal ray. We will denote the set of rational contractible curves by $\mathrm{Cont}(X)$.
 \vspace{0.2cm}

To conclude this prelude, we recall that the divisorial Zariski decomposition of pseudo-effective divisors on projective HK manifolds is characterized with respect to the BBF form (see Subsection \ref{divZardecomp}). Furthermore, as in the case of surfaces, the positivity of any pseudo-effective divisor on $X$ is encoded by the positive part of this decomposition. Hence, if $D$ is a big divisor on $X$, one may expect to find a characterization of $\mathbf{B}_+(D),\mathbf{B}_-(D),\mathbf{B}(D)$ and of the stable classes, using the positive part and the negative part of the divisorial Zariski decomposition of $D$, together with the intersections of $D$ with the rational curves contained in the variety.

\begin{maintheorem}\label{Thm:baselociHK}
Let $X$ be a projective hyper-Kähler manifold and $D$ a big $\mathbf{R}$-divisor on $X$. 
\begin{enumerate}
\item All the irreducible components of $\mathbf{B}_+(D),\mathbf{B}_-(D)=\mathbf{B}(D)$ are algebraically coisotropic. In particular, they have dimension at least $\mathrm{dim}(X)/2$ (whenever the asymptotic base loci are not empty). 

\item The divisorial irreducible components of $\mathbf{B}_+(D)$ are the prime exceptional divisors which are $q_X$-orthogonal to $P(D)$. For any irreducible component $V$ of $\mathbf{B}_+(D)$ there exist rational, contractible curves covering $V$, and
 intersecting $P(D)$ negatively, or trivially, on some birational hyper-Kähler model of $X$. 
 
 \item The divisorial irreducible components of $\mathbf{B}_-(D)$ are the prime exceptional divisors supporting $N(D)$. For any non-divisorial irreducible component of $\mathbf{B}_-(D)= \mathbf{B}(D)$ there exist rational, contractible curves covering the component, and intersecting $P(D)$ negatively on some birational hyper-Kähler model of $X$. 
  \item $P(D)$ is unstable if and only if $\mathbf{B}_+(P(D))$ contains a curve $C\in \mathrm{Cont}(X)$, such that $C \not \subset \mathrm{Ind}(f)$ for some birational map $f \colon X \DashedArrow X'$ of projective HK manifolds, with $f_*(P(D))$ nef, and $f_*(P(D)) \cdot f_*(C)=0$. Furthermore, all the destabilizing numbers are rational numbers.
  \end{enumerate}
\end{maintheorem}

Let us state one particular consequence of item (1) of the theorem above, that should be of independent interest.

\begin{corollary}
    Let $f:X\dashrightarrow X'$ be a birational map between projective HK manifolds. Then (any irreducible component of) the indeterminacy locus of $f$ is algebraically coisotropic.
\end{corollary}

\begin{Rem*}
Following the notation of Theorem \ref{Thm:baselociHK}, let $D$ be any big divisor on $X$, $V$ an irreducible component of $\mathbf{B}_+(D)$ (or $\mathbf{B}_-(D)=\mathbf{B}(D)$), and $\{C_t\}_t$ rational curves covering $V$, like those in the statement. Then, the birational model over which $C_t$ can be contracted might depend on $t$. Furthermore, we provide an example showing that, in general, we cannot avoid considering the birational hyper-Kähler models of the variety (cf.\ Subsection \ref{exampleHT}).
\end{Rem*}

We now introduce some more terminology, to state the second main result of this paper.
\vspace{0.2cm}

Let $Y$ be any smooth complex projective variety.  The cone $\mathrm{Amp}_k(X)$ (where $k$ ranges from 1 to $\mathrm{dim}(Y)$) is the convex cone of divisor classes $\alpha$ on $Y$ with $\mathrm{dim}\left(\mathbf{B}_+(\alpha)\right)<k$.  It is known that $\overline{\mathrm{NE}(Y)}^{\vee}$, the dual of the Kleiman-Mori cone $\overline{\mathrm{NE}(Y)}$, is equal to the nef cone $\mathrm{Nef}(Y)=\overline{\mathrm{Amp}(Y)}=\overline{\mathrm{Amp}_1(Y)}$ (see \cite[Theorem 1.1]{Ein1} for the last equality). This classical duality result is known as "Kleiman's criterion for amplitude" (cf.\ \cite[Theorem 1.4.29]{Laz}), and it holds for any projective (non-necessarily smooth) variety. Later, Boucksom, Demailly, Păun, and Peternell proved that the dual of the pseudo-effective cone $\overline{\mathrm{Eff}(Y)}=\overline{\mathrm{Amp}_{d}(Y)}$ (where $d=\mathrm{dim}(Y)$) is equal to the cone of mobile curves $\overline{\mathrm{Mob}(Y)} \subset \overline{\mathrm{NE}(Y)}$, where $\mathrm{Mob}(Y) \subset \mathrm{NE}(Y)$ is the cone spanned by the classes of curves moving in a family covering $Y$ (cf.\ \cite[Theorem 2.2]{BDPP}). Sam Payne proved that on a complete, $\mathbf{Q}$-factorial toric variety $T$, the closure of the cone $\mathrm{Amp}_k(T)$ is dual to the cone $\mathrm{bMob}_k(T)$ of curves which are birationally $k$-mobile, i.e.\ the closure of the cone spanned by classes which on some small $\mathbf{Q}$-factorial modification $T'$ of $T$ are represented by curves sweeping out the birational transform of $k$-dimensional subvarieties of $T$ (cf.\ \cite[Theorem 2]{Payne2005}). Furthermore, Payne asked whether such a result could have been generalized to other classes of varieties. Payne's result for toric varieties was generalized by Choi (cf.\ \cite[Theorem 1.1, Corollary 1.1]{Choi2012}) to Mori dream spaces and $\mathbf{Q}$-factorial FT varieties (see \cite[Definition 4.4]{Choi2012} for the definition of FT variety). 
\vspace{0.2cm}

There are two obstructions for projective HK manifolds to be Mori dream spaces: the group of birational self maps could be infinite, and, in general, it is not known whether a nef, integral divisor which is not big is semiample. The latter condition is always satisfied for the known deformation classes of HK manifolds, hence a projective HK manifold belonging to one of the known deformation classes is a Mori dream space if and only if its group of birational self maps is finite (see \cite[Corollary 3.11]{Den2}), and Choi's result applies in these cases. Theorem \ref{Thm:duality} gives an affirmative answer to the question asked by Sam Payne (actually, as it is mentioned in Payne's article, it seems that the original question was raised by Robert Lazarsfeld and Olivier Debarre) in the case of arbitrary projective HK manifolds.
\vspace{0.2cm}

We are now ready to state the second main result of this paper.

\begin{maintheorem}\label{Thm:duality}
    Let $X$ be a projective HK manifold of dimension $2n$. Then
    \begin{equation}
        \overline{\mathrm{Amp}_k(X)}=\mathrm{bMob}_k(X)^{\vee},
    \end{equation}
    for any $1<k<2n$, where $\mathrm{dim}(X)=2n$.
    Furthermore, if $k=1,\dots,n$, we have $\overline{\mathrm{Amp}_k(X)}=\mathrm{Nef}(X)=\overline{\mathrm{NE}(X)}^{\vee}$.
\end{maintheorem}

It is worth observing that if $Y$ is any smooth, complex projective variety, a slightly weaker description for the dual cones $\overline{\mathrm{Amp}_k(Y)}^{\vee}$, and for the asymptotic base loci of pseudo-effective divisors has been provided by Brian Lehman (cf.\ \cite{Lehmann2011TheMC}). In particular, items (1) and (2) of Theorem \ref{Thm:baselociHK}, and Theorem \ref{Thm:duality}, are a strengthening of \cite[Theorem 1.5, Proposition 3.2, Corollary 3.3]{Lehmann2011TheMC} (in the case of projective HK manifolds).
\subsection*{\large \itshape Organization of the paper}

\begin{enumerate}
\item In Section \ref{Section1} we give the basic notions and results that will be used to achieve our goals.

\item In Section \ref{Section2} we carefully study the asymptotic base loci of big divisors on projective HK manifolds. Moreover, we prove the first three items of Theorem \ref{Thm:baselociHK}.

\item In Section \ref{Section3}, we study how the asymptotic base loci vary when slightly perturbing a big divisor class in the big cone. This leads to the decomposition of the big cone of projective HK manifolds into stability chambers. Furthermore, we prove the fourth item of Theorem \ref{Thm:baselociHK}.

\item In Section \ref{duality} we prove Theorem \ref{Thm:duality}.

\item In Section \ref{Section6} we provide for the effective cone of a projective HK manifold a decomposition into Mori-type chambers (analogous to that in Mori chambers for Mori dream spaces), resuming the birational geometry of the variety.

\item Section \ref{Examples} is devoted to examples.

\item In Section \ref{Section8} we ask two questions that would be interesting to answer.

\end{enumerate}

\section{\bf \scshape Preliminaries}\label{Section1}
In this section, we provide the background needed to prove the results of this paper. Let $Y$ be a normal complex projective variety of complex dimension $d$, and $X$ a projective HK manifold of dimension $2n$. Any divisor will be tacitly assumed to be Cartier.

\subsection{\large \itshape Divisors and cones}
We recall that an $\mathbf{R}$-divisor on $Y$ is a (finite) formal linear combination with real coefficients of integral divisors. The Néron-Severi space of $Y$ is the group of $\mathbf{R}$-divisors modulo the numerical equivalence relation. It will be denoted by $N^1(X)_{\mathbf{R}}$.
 Two (integral, rational, real) divisors $D,D'$ on $Y$ are numerically equivalent (and in this case we write $D\equiv D'$), if $D\cdot C=D'\cdot C$, for any irreducible curve $C$ in $Y$.
 \vspace{0.2cm}
 
Two $\mathbf{R}$-divisors $D,D'$ on $Y$ are $\mathbf{R}$-linearly equivalent (and in this case we write$D \sim_{\mathbf{R}} D'$) if $D-D'$ is an $\mathbf{R}$-linear combination of principal divisors.
 \vspace{0.2cm}

An integral divisor $D$ on $Y$ is movable if $\mathbf{B}(D)$, the stable base locus of $D$, has no divisorial irreducible components. The movable cone of $Y$ (in symbols $\overline{\mathrm{Mov}(Y)}$) is the closure in the Néron-Severi space of the cone spanned by movable, integral divisor classes. Any divisor class lying in $\overline{\mathrm{Mov}(Y)}$ will be called a "movable class". Notice that, for a projective hyper-Kähler manifold $X$, we have, $\overline{\mathrm{Mov}(X)}=\overline{\mathscr{BK}_X}\cap N^1(X)_{\mathbf{R}}$ (cf.\ \cite[Theorem 7]{Hassett}), where $\mathscr{BK}_X$ is the birational Kähler cone of $X$, i.e.\
\[
\mathscr{BK}_X=\bigcup_{f\colon X \DashedArrow X'} f^{*}\mathscr{K}_{X'},
\]
where $\mathscr{K}_{X'}$ is the Kähler cone of another hyper-Kähler manifold $X'$, $f\colon X \DashedArrow X'$ is a birational map and the closure $\overline{\mathscr{BK}_X}$ is taken in $H^{1,1}(X,\mathbf{R})$. Another useful characterization of the movable cone of $X$
 is the following
 \begin{equation}\label{Characterization:Movablecone}
 \overline{\mathrm{Mov}(X)}=\{D\in N^1(X)_{\mathbf{R}} \; | \; q_X(D,E)\geq 0 \text{ for any prime divisor } E\}.     
 \end{equation}

For a proof of \ref{Characterization:Movablecone}, see for example \cite[Lemma 2.7, Remark 2.10]{Den1}. An integral divisor $D$ on $Y$ is big if $h^0(Y,\mathscr{O}_Y(mD))$ grows like $m^d$ (recall that $d$ is the dimension of $Y$), for $m\gg 0$. An $\mathbf{R}$-divisor on $Y$ is big if it can be written as a linear combination of big integral divisors, with positive real coefficients. The big cone of $Y$ is the cone in $N^1(Y)_{\mathbf{R}}$ spanned by the big divisor classes. It will be denoted by $\mathrm{Big}(Y)$ and its closure in $N^1(X)_{\mathbf{R}}$ is known as the pseudo-effective cone of $Y$. Usually, the pseudo-effective cone of $Y$ is denoted by $\overline{\mathrm{Eff}(Y)}$, because it can be equivalently defined as the closure of the cone spanned by effective integral divisor classes (namely $\mathrm{Eff}(Y)$), which is known as the effective one of $Y$.
Actually, the big cone is the interior of both the effective cone and the pseudo-effective cone. For details about divisors and cones in the Néron-Severi space, we refer the reader to \cite{Laz}.

\begin{Rem}
By \ref{Characterization:Movablecone}, we have that $\overline{\mathrm{Eff}(X)}$ and $\overline{\mathrm{Mov}(X)}$ are one the dual of the other with respect to the BBF pairing.
\end{Rem}

\subsection{\large \itshape Asymptotic base loci}
We recall the definition of the various asymptotic base loci and provide some basic properties of them. We refer the reader to \cite{Ein2} for a comprehensive treatment of the argument.

\begin{Def}Let $D$ be an $\mathbf{R}$-divisor on $Y$.
\begin{itemize}
\item 
The real \textit{stable base locus} of $D$ is
\[
\mathbf{B}(D):=\bigcap\{\mathrm{Supp}(E)\;|\; E \text{ effective $\mathbf{R}$-divisor, $E \sim_{\mathbf{R}} D$ }\}.
\]
If $D$ is a $\mathbf{Q}$-divisor, the real stable base locus of $D$ coincides with the usual one (cf.\ \cite[Proposition 1.1]{BBP2013}).

\item The \textit{augmented base locus} of $D$ is
\[
\mathbf{B}_{+}(D):=\bigcap_{D=A+E}\mathrm{Supp}(E),
\]
where the intersection is taken over all the decompositions of the form $D=A+E$, where $A$ and $E$ are $\mathbf{R}$-divisors such that $A$ is ample and $E$ is effective.

\item The \textit{restricted base locus} of $D$ on $Y$ is
\[
\mathbf{B}_{-}(D):=\bigcup_{A}\mathbf{B}(D+A),
\]
where the union is taken over all ample divisors $A$. 
\end{itemize}
\end{Def}

If $D$ is any divisor on $Y$, we have the inclusions $\mathbf{B}_-(D) \subset \mathbf{B}(D)\subset \mathbf{B}_+(D)$ (cf.\ \cite[Example 1.16]{Ein2}, and \cite[Introduction]{BBP2013}), which in general are strict, as it is shown by the following example. 
\begin{Ex}
Consider the blow-up of the plane at 2 points, $L$ be any line not passing through the two points, and $L'$ its strict transform. Furthermore, let $E_1, E_2$ be the two exceptional divisors arising from the 2 points. Consider the divisor $D_i:=L'+E_i$, for $i=1,2$. We observe that $\mathbf{B}_{-}(D_i)=\mathbf{B}(D_i)=E_i$, while $\mathbf{B}_+(D_i)=E_1\cup E_2$ (it follows directly from \cite[Examples 1.10, 1.11, 3.4]{Ein2}). For a strict (non-trivial) inclusion $\mathbf{B}_{-} \subsetneq \mathbf{B}$, with $\mathbf{B}_{-}$ not algebraic, one can consider the beautiful example provided by Lesieutre (cf.\ \cite{lesieutre}). For an example with $\emptyset \neq \mathbf{B}_{-} \subsetneq \mathbf{B} \subsetneq \mathbf{B}_{+} \subsetneq Y$ see \cite[Example 3.4, (iii)]{TX23}.
\end{Ex}
\vspace{0.2cm}

\begin{Rem}
The augmented and restricted base loci on $Y$ can be safely studied in the Néron-Severi space because they do not depend on the representative of the class we have chosen. In general, we cannot say the same for the (real) stable base locus. However, on HK manifolds the numerical and $\mathbf{R}$-linear equivalence relations coincide, and this implies that also the real stable base locus can be safely studied up to numerical equivalence. Indeed, suppose that $D, D'$ are numerically equivalent $\mathbf{R}$-divisors on $X$, then $D-D'=\sum_ia_iD_i$, where $a_i \in \mathbf{R}$ and $D_i$ is a numerically trivial integral divisor, and hence also linearly equivalent to $0$. In particular, $D_i$ is a principal divisor, thus $D \sim_{\mathbf{R}} D'$.
\end{Rem}

Note that if $D$ is any $\mathbf{R}$-divisor on $Y$, $\mathbf{B}_{+}(D) \subsetneq Y$ if and only if $D$ is big, $\mathbf{B}_{-}(D)\subsetneq Y$ if and only if $D$ is pseudo-effective (i.e.\ the class of $D$ in $N^1(Y)_{\mathbf{R}}$ is the limit of some sequence of big classes), and $\mathbf{B}_-(D)=\emptyset$ if and only if $D$ is nef. Moreover, for a $\mathbf{Q}$-divisor $D$, $\mathbf{B}(D)=\emptyset$ if and only if (some positive multiple of) $D$ is semiample.
\vspace{0.2 cm}

As remarked in the introduction, in general, it is not possible to describe the asymptotic base loci of a big divisor using the intersections of the divisor with the curves, as the following example shows.

\begin{Ex}\label{Ex:notnumdet} Let $Q$ be a point of $\mathbf{P}^2$ and $S':=\mathrm{Bl}_Q(\mathbf{P}^2)$. Now, let $S$ be the blow up of $S'$ at a point $P$ lying on the exceptional divisor $E'$ of $S'$. If $F$ is the exceptional divisor of $S$ (i.e.\ the fiber over $P$), and $E$ is the strict transform of $E'$ on $S$, we have $F^2=-1$, and $E^2=-2$. Furthermore, the Gram matrix of $\{E, F\}$ is negative definite. Let $D$ be a big and nef divisor that is orthogonal to both $E, F$, and consider the divisors $D_1:=D+E+F$, $D_2:=D+E+3F$. Then $\mathbf{B}_-(D_2)=\mathbf{B}_-(D_1)=\mathrm{Supp}(E)\cup\mathrm{Supp}(F)$. But $D_1\cdot E=-1<0$, $D_1\cdot F=0$, and $D_2 \cdot E =1>0$. 
\end{Ex}

\subsection{\large \itshape Divisorial Zariski decomposition}\label{divZardecomp}
On any compact complex manifold we have the divisorial Zariski decomposition, which was introduced by Boucksom in his fundamental paper \cite{Boucksom} (but see also Nakayama's construction, cf.\ \cite{Nakayama}), and is a generalization of the classical Zariski decomposition on surfaces. The divisorial Zariski decomposition will occupy a central role in this paper, and in the case of HK manifolds Boucksom characterized it with respect to the BBF form.

\begin{Thm}[Theorem 4.8, \cite{Boucksom}]\label{DivZarDec}
Let $D$ be a pseudo-effective $\mathbf{R}$-divisor on $X$. Then $D$ admits a divisorial Zariski decomposition, i.e.\ we can write $D=P(D)+N(D)$ in a unique way such that:
\begin{enumerate}
\item $P(D) \in \overline{\mathrm{Mov}(X)}$.
\item $N(D)$ is an effective and (if non-zero) exceptional $\mathbf{R}$-divisor (i.e.\ the Gram matrix of its irreducible components is negative definite), 
\item $P(D)$ is orthogonal (with respect to $q_X$) to any irreducible component of $N(D)$.
\end{enumerate}
The divisor $P(D)$ (resp.\ $N(D)$) is the \textit{positive part} (resp.\ \textit{negative part}) of $D$.
\end{Thm}

\begin{Rem}\label{Rem:zardecomp}
 If the divisor $D$ is effective, its positive part can be defined as the maximal $q_X$-nef subdivisor of $D$. This means that, if $D=\sum_ia_iD_i$, where the $D_i's$ are the irreducible components of $D$, the positive part of $D$ is the maximal divisor $P(D)=\sum_ib_iD_i$, with respect to the relation "$P(D)\leq D$ if and only if $b_i \leq a_i$ for every $i$", such that $q_X(P(D),E)\geq 0$, for any prime divisor $E$. 
\end{Rem}
 
 The positive part of the divisorial Zariski decomposition of a big divisor encodes most of the positivity of the divisor itself. 
Indeed, given any big $\mathbf{Q}$-divisor $D$ on $X$, one has $H^0(X,mD)\cong H^0(X,mP(D)) $, for $m$ sufficiently divisible (cf.\ \cite[Theorem 5.5]{Boucksom}), and also we have $\mathbf{B}_+(D)=\mathbf{B}_+(P(D))$, $\mathbf{B}_-(D)=\mathbf{B}_-(P(D))$  (see Lemma \ref{lem:B+positivepart}). Furthermore, $D$ is big if and only if $P(D)$ is (see \cite[Proposition 3.8]{Boucksom}).

\subsection{\large \itshape Wall divisors and rational curves}
Recall that $H_2(X,\mathbf{Z}) \cong H^2(X,\mathbf{Z})^{*} \subset H^2(X,\mathbf{Q})$. Hence, if $\gamma$ is a curve class in $H_2(X,\mathbf{Z})$, up to numerical equivalence, there exists a unique $\mathbf{Q}$-divisor $D$  on $X$ such that $\alpha \cdot \gamma=q_X(\alpha,D)$, for any $\mathbf{R}$-divisor class $\alpha$. The square of $\gamma$ with respect to the BBF quadratic form is defined as $q_X(\gamma):=q_X(D)\in \mathbf{Q}$.

\begin{Def}
A wall divisor on $X$ is an integral divisor $D$ such that $q_X(D)<0$ and $g(D)^{\perp} \cap \mathscr{BK}_X=\emptyset$, for any $g\in \mathrm{Mon}^2_{\mathrm{Hdg}}(X)$.
\end{Def}

Since for a wall divisor $D$ it holds $q_X(D)<0$, we have that $g(D)^{\perp}$ intersects $\mathscr{C}_X$. An important property of wall divisors is that they stay wall divisors under parallel transport (cf.\ \cite[Theorem 1.3]{Mongardi1}). Furthermore, there is an important relation between them and (rational) extremal curves on $X$ and its birational models.

\begin{Prop}[Proposition 1.5, \cite{Mongardi1}]
Let $h$ be an ample class on $X$. There is a bijective correspondence between primitive wall divisors intersecting positively $h$ and negative extremal rays of the Mori cone of $X$ (with respect to the BBF pairing).
\end{Prop}
Then, a wall divisor, up to the action of $\mathrm{Mon}_{\mathrm{Hdg}}^2(X)$, is dual to some (rational) extremal curve on some birational model of $X$.
\begin{Rem}
The MBM classes defined in \cite{AmVerb} are exactly the homology classes which are dual to wall divisors (cf.\ \cite[Proposition 2.3, Remark 2.4]{KLM})
\end{Rem} 

\section{\bf \scshape Characterization of Asymptotic Base Loci}\label{Section2}
This section is devoted to the study of the asymptotic base loci of big divisors on projective HK manifolds.
$X$ will always denote a projective HK manifold of complex dimension 2n.

\begin{Prop}\label{positivepart}
Let $D$ be any big $\mathbf{R}$-divisor on $X$. Then $\mathbf{B}_+(D)=\mathbf{B}_+(P(D)) \cup \mathrm{Supp}(N(D))$, and $\mathbf{B}_-(D)=\mathbf{B}_-(P(D)) \cup \mathrm{Supp}(N(D))$.
\begin{proof}
 Write $D=A+N$, where $A$ and $N$ are ample and effective respectively. Then $P(D)=A+N'$, where $N' \leq N$ is effective (see Remark \ref{Rem:zardecomp}), and so $\mathrm{Supp}(N')\subset \mathrm{Supp}(N)$. This implies that $\mathbf{B}_+(P(D)) \subset \mathbf{B}_+(D)$. Moreover $\mathrm{Supp}(N(D)) \subset \mathbf{B}_+(D)$, hence $\mathbf{B}_+(P(D)) \cup \mathrm{Supp}(N(D)) \subset \mathbf{B}_+(D)$. On the other hand, writing $P(D)=A'+N'$, with $A'$ and $N'$ ample and effective respectively, we obtain that $D=A'+N'+N(D)$, hence $\mathbf{B}_+(D)\subset \mathbf{B}_+(P(D)) \cup \mathrm{Supp}(N(D))$, and we are done. By definition, 
 \[
 \mathbf{B}(D)=\bigcap_{E \equiv D} \mathrm{Supp}(E),
 \]
 and we know by \cite[Theorem A]{BBP2013} that, for any big divisor $D$, $\mathbf{B}_{-}(D)=\mathbf{B}(D)$. If we write $D=P(D)+N(D)$, the above equality becomes 
  \[
 \mathbf{B}_{-}(D)=\bigcap_{E \equiv D} \mathrm{Supp}(P(E)+N(E)),
 \]
 and for any effective divisor $E$ with $E \equiv D$, we have $\mathrm{Supp}(N(E))=\mathrm{Supp}(N(D))$ (cf.\ \cite[Remark 1.5]{Den1}), so that 
   \[
 \mathbf{B}_{-}(D)=\left(\bigcap_{E \equiv D} \mathrm{Supp}(P(E))\right)\cup \mathrm{Supp}(N(D)).
 \]
 If $E'$ is en effective divisor numerically equivalent to $P(D)$, $E'+N(D)\equiv P(D)+N(D)$. On the other hand, if $E$ is an effective divisor numerically equivalent to $D$, $P(E)$ is numerically equivalent to $P(D)$ (cf.\ \cite[Remark 1.5]{Den1}). This implies that 
    \[
 \mathbf{B}_{-}(D)=\left(\bigcap_{E \equiv P(D)} \mathrm{Supp}(E)\right)\cup \mathrm{Supp}(N(D)).
 \]
 But $\bigcap_{E \equiv P(D)} \mathrm{Supp}(E)=\mathbf{B}(P(D))=\mathbf{B}_{-}(P(D))$, and this concludes the proof.
\end{proof}
\end{Prop}

In particular, to study the asymptotic base loci of big divisors on projective HK manifolds, we can restrict ourselves to big, movable divisors.

\begin{Def}
Given a big divisor $D$ with class lying in $\overline{\mathrm{Mov}(X)}$, we define $\mathrm{Null}_{q_X}(D)$ as the set of prime divisors which are $q_X$-orthogonal to $D$.
\end{Def}

\begin{Rem}
Note that for any big and $q_X$-nef divisor $D$, $\mathrm{Null}_{q_X}(D)$ is made of prime exceptional divisors. Indeed, by the bigness of $D$, we can write $D=A+N$ (cf.\ \cite[Proposition 2.2.22]{Laz}), with $A$ ample and $N$ effective $\mathbf{R}$-divisors. If $q_X(D,E)=0$ for some prime divisor $E$, then $E$ must be an irreducible component of $N$ and $q_X(E)<0$, because $q_X(A,E)>0$, as $A$ is ample.
\end{Rem}

\begin{Prop}\label{lem:B+positivepart}
Let $D$ be a big $\mathbf{R}$-divisor. Then
\[
\mathbf{B}_+(D) = \mathbf{B}_+(P(D))
\]
and the divisorial part $\mathbf{B}_+(D)_{\mathrm{div}}$ of $\mathbf{B}_+(D)$ is $\mathrm{Null}_{q_X}(P(D))$.
\end{Prop}
\begin{proof}
By Proposition \ref{positivepart}, as $\mathrm{Supp}(N(D))\subset \mathrm{Null}_{q_X}(P(D))$, we only need to check that \[\mathbf{B}_+(P(D))_{\mathrm{div}}=\mathrm{Null}_{q_X}(P(D)).\] Suppose $E \not \subset \mathbf{B}_{+}(P(D))$, then $q_X(P(D),E)>0$. Indeed,
\[
\mathbf{B}_{+}(D):=\bigcap_{\substack{D=A+N \\ A \text{ ample } \\ N \text{ effective}}} \mathrm{Supp}(N),
\] 
hence there exists a decomposition $D=A+N$ such that $E\not \subset \mathrm{Supp}(N)$. As $q_X(A,E)>0$ (cf.\ \cite[1.11]{Huybrechts1999}) and $q_X$ is an intersection product we are done. This implies that if $E$ is any prime divisor satisfying $q_X(P(D),E)=0$, then $E \subset \mathbf{B}_{+}(P(D))$, so that $\mathbf{B}_{+}(P(D)) \supset \mathrm{Null}_{q_X}(P(D))$. Now, let $E$ be any prime divisor lying in $\mathbf{B}_{+}(P(D))$. By \cite[Theorem 1.2]{MatsushitaZhang2013} there exists a birational model $X'$ of $X$ and a birational map $f \colon X \DashedArrow X'$ such that $f_{*}(P(D))$ is big and nef. Furthermore, arguing as in \cite[Proof of Theorem A]{BBP2013}, we obtain that the strict transform $E'$ of $E$ via $f$ is an irreducible component of $\mathbf{B}_{+}(f_{*}P(D))$. But then, by \cite[Theorem 5.2, item (b)]{Ein1} and \cite[Example 5.5]{Ein1}, we have $(f_{*}P(D))^{2n-1}\cdot E'=0$, and we have (cf.\ \cite[Exercise 23.2]{GHJ2003})
\[
q_X(f_{*}P(D),E')\int (f_{*}P(D))^{2n}=q_X(f_{*}P(D))\int(f_{*}P(D))^{2n-1}\cdot E'=0,
\]
which forces $q_X(f_{*}P(D),E')=0$, because $f_{*}(P(D))$ is big and nef, and hence 
\[
\int (f_{*}P(D))^{2n} > 0.
\] 
But then also $q_X(P(D),E)=0$, because $f_{*}$ is an isometry (with respect to the BBF forms), hence we conclude that $E \in \mathrm{Null}_{q_X}(P(D))$.\end{proof}

\begin{Rem}
We observe that if $D$ lies in $\mathrm{int}(\mathrm{Mov}(X))$, then $\mathbf{B}_+(D)$ does not contain divisorial irreducible components. Indeed, by assumption $q_X(D,E)>0$ for any prime exceptional divisor $E$, hence $\mathbf{B}_+(D)$ does not have any divisorial irreducible component, by the above proposition.
\end{Rem} 

\begin{Def}\label{defMBM}
Let $D$ be any big $\mathbf{R}$-divisor, and fix a log MMP for the pair $(X,\epsilon P(D))$, with $0<\epsilon \ll 1$. We define $\mathrm{MBM}(D)$ to be the set of rational curves in $\mathrm{Cont}(X)$ which are flipped while running the chosen log MMP for $(X,\epsilon P(D))$, or contracted at the end of the chosen log MMP for $(X,\epsilon P(D))$. 
\end{Def}

\begin{Rem}
Let us explain Definition $\ref{defMBM}$. Given any big $\mathbf{R}$-divisor $D$ on $X$, a curve $C \in \mathrm{MBM}(D)$ satisfies one of the two following possibilities:

$\bullet$ There exists a birational map $f\colon X \DashedArrow X'$, arising from the chosen log MMP for the pair $(X,\epsilon P(D))$, with $C \not \subset \mathrm{Ind}(f)$, such that $C'=f_{*}(C)$ (the strict transform of $C$ via $f$) is flipped at the next step of the chosen log MMP.

$\bullet$ The curve $C$ "survives" until the end of the chosen log MMP. Let $f\colon X \DashedArrow X'$ be a model on which $f_*(P(D))=P(D')$ (where $D'=f_*(D)$) is nef (i.e.\ where the chosen log MMP terminates). In this case  $f_*(C)=C'$ is contracted by the morphism induced by a big and nef integral divisor class lying in the relative interior of the minimal (with respect to the dimension) extremal face of the nef cone of $X'$ containing the class of $P(D')$. Clearly, in this case $P(D')$ is not ample.
\end{Rem} 

A priori, the elements of $\mathrm{MBM}(D)$ could depend on the log MMP we have chosen and any time we write $\mathrm{MBM}(D)$, this has to be thought of with respect to a fixed log MMP for $(X,\epsilon P(D))$.

\begin{Rem}\label{Rem:leqBplus}
   Let $D$ be any big Cartier $\mathbf{R}$-divisor on any normal complex projective variety $Y$. We observe that if $C$ is an irreducible curve with $C \not \subset \mathbf{B}_+(D)$ then $D \cdot C >0$. Indeed, by definition of $\mathbf{B}_+(D)$, we can write $P=A+E$, where $A$ is Cartier and ample, $E$ is Cartier and effective, and $C \not \subset \mathrm{Supp}(E)$. It follows that if $D\cdot C \leq 0$ for some irreducible curve $C$, $C \subset \mathbf{B}_+(D)$.
   Note that if $C$ is a curve lying in $\mathbf{B}_+(D)$, it is not true that $D\cdot C\leq 0$, in general, also when $[D]\in \mathrm{int}(\overline{\mathrm{Mov}(X)})$ (see the subsection "The example of Hassett and Tschinkel" for a counterexample). 
\end{Rem}

The following proposition is a consequence of \cite[Theorem 1.2]{MatsushitaZhang2013}, \cite[Proof of Theorem A]{BBP2013}, and \cite[Theorem 1]{Kaw91}. It gives a recipe to compute the augmented base locus of any big divisor on a projective HK manifold. The main limitation of this method is that to compute the augmented base locus of a big divisor one must know how to run a log MMP for the positive part of the divisor.

\begin{Prop}\label{Prop:B+MBM}
Let $D$ be any big $\mathbf{R}$-divisor on a projective HK manifold $X$, then
\[
\mathbf{B}_+(D) = \bigcup_{C \in \mathrm{MBM}(P(D))} \mathrm{Supp}(C).
\]
\end{Prop}
\begin{proof}
 Since $\mathbf{B}_+(D) = \mathbf{B}_+(P(D))$, we can assume that $D$ is movable (i.e.\ its class lies in $\mathrm{Mov}(X)$).
    
By running the log MMP associated with $\mathrm{MBM}(D)$, for the pair $(X,\varepsilon D)$, for sufficiently small $\varepsilon \in \mathbf{Q}$, we obtain a finite sequence of log flips (cf.\ \cite[Theorem 1.2]{MatsushitaZhang2013})
    \begin{equation}\label{MMP}
        \begin{tikzcd}
            X_0 = X\ar[rr,dashed, "\varphi_0"]\ar[dr,swap,"\pi_0^-"]& & X_1\ar[rr,dashed,"\varphi_1"] \ar[dl,"\pi_0^+"] \ar[dr,swap,"\pi_1^-"]  & & X_2 \ar[dl,"\pi_1^+"] \ar[r,dashed]& \dots \ar[r,dashed] & X_{n-1}\ar[rr,dashed, "\varphi_{n-1}"] \ar[dr,swap,"\pi_{n-1}^-"] && X_n \ar[dl,"\pi_{n-1}^+"]\\
             & Z_0 & & Z_1 & & & &Z_{n-1}
        \end{tikzcd}
    \end{equation}

    such that, if we set $D_0 = D$, $D_i:= \varphi_{i-1,*} D_{i-1}$, and 
    \[
    \varphi := \varphi_{n-1}\circ\cdots \circ \varphi_0,
    \]
    for any $i=1,\dots,n$, $\varphi_{*}(D)$ is big and nef on $X_n$. Without loss of generality, we may assume that $D_n$ is not ample. So, let $f\colon X_n \to Y$ be the birational morphism contracting the extremal face $D_n^{\perp}\cap \overline{\mathrm{NE}(X)}$ of the Mori cone $\overline{\mathrm{NE}(X)}$ (\cite[Theorem 3.7]{KM98}). Now, choose a point $x_0$ in $\mathbf{B}_+(D)$, and define recursively $x_i:= \varphi_{i-1} (x_{i-1})$ (whenever we can do it), for any $i=1,\dots,n-1$. Suppose that $x_0\in \mathrm{Exc}(\pi^{-}_0)$. Then, by \cite[Theorem1]{Kaw91} there exists a rational curve $C$ passing through $x_0$ that is contracted by $\pi_0$. Clearly $D\cdot C <0$, hence $C \subset \mathrm{MBM}(D)$.
    
    If $x_0 \notin \mathrm{Exc}(\pi^{-}_0)$, let $V_0$ be an irreducible component of $\mathbf{B}_{+}(D)$ containing it. Arguing as in \cite[Proof of Theorem A]{BBP2013}, we see that the strict transform $V_1$ of $V_0$ via $\varphi_0$ is an irreducible component of $\mathbf{B}_{+}(D_1)$. 
    If $x_1 \in \mathrm{Exc}(\pi^{-}_1)$, arguing as above, we find a rational curve $C_1$ on $X_1$ containing $x_1$ that is contracted by $\pi_1^{-}$. Moreover, we have $D_1\cdot C_1 <0$. The strict transform $C$ of this curve via $\phi_0^{-1}$ is a rational curve passing through $x_0$ which belongs to $\mathbf{B}_+(D)$. It is important to point out that it could hold $D \cdot C \geq 0$, even though $D_1\cdot C_1 <0$. If  $x_1 \notin \mathrm{Exc}(\pi^{-}_1)$, the strict transform $V_2$ of $V_1$ via $\varphi_1$ is an irreducible component of $\mathbf{B}_+(D_2)$ and we can continue like this for finitely many times. If it never happens that $x_i \in \mathrm{Exc}(\pi_i^{-})$, for any $i=1,\dots,n-1$, then $x_n \in \mathrm{Exc}(f)=\mathbf{B}_+(D_n)$, and $D_n$ is nef. The face $D_n^{\perp}\cap \overline{\mathrm{NE}(X_n)}$ which is $(D_n-\epsilon A)$-negative (and $D_n-\epsilon A$ is big for $\epsilon\ll 0$), is contracted by $f$. Now, we choose $0<\epsilon \ll 1$ such that $D-\epsilon A$ is a big $\mathbf{Q}$-divisor.  Using the Kleiman's criterion for $f$-ampleness (cf.\ \cite[Theorem 1.44]{KM98}), we see that $-(D_n-\epsilon A)$ is $f$-ample. Up to a rescaling, we can assume that $(X,D_n-\epsilon A)$ is klt. Then, using again \cite[Theorem 1]{Kaw91}, we conclude that through $x_n$ passes a rational curve $C_n$ which is contracted by $f$, and $D_n\cdot C_n=0$. We conclude that the strict transform $C$ of $C_n$ via $\varphi^{-1}$ is a rational curve passing through $x_0$, contained in $\mathbf{B}_+(D)$. This shows that $\mathbf{B_+}(D)\subset \mathrm{MBM}(D)$. 
     
    On the other hand, let $C$ be a curve lying in $\mathrm{MBM}(D)$. Suppose that at a certain point the strict transform $C_i$ of $C$ via $\varphi_{i-1}$ is contracted by $\pi^{-}_i$. Then $C_i \subset \mathbf{B}_+(D_i)$. Let $V_i$ be the irreducible component of $\mathbf{B}_+(D_i)$ containing $C_i$. Arguing as in \cite[Proof of Theorem A]{BBP2013}, we conclude that the strict transform $V$ of $V_i$ via $(\varphi_{i-1} \circ \cdots\circ \varphi_1 \circ \varphi)^{-1}$ is an irreducible component of $\mathbf{B}_+(D)$, and hence $C \subset \mathbf{B}_+(D)$. Otherwise, $D_n \cdot C_n=0$, because $D_n$ is nef and $C_n$ can be contracted. But then $C_n \subset \mathbf{B}_+(D_n)$, by Remark \ref{Rem:leqBplus}, and arguing once again as in \cite[Theorem A]{BBP2013}, we obtain that $C \subset \mathbf{B}_+(D)$ and $\mathrm{MBM}(D) \subset \mathbf{B}_+(D)$, and we are done.
\end{proof}

\begin{Rem}
Let $D$ be any big $\mathbf{R}$-divisor on $X$. It is important to point out that if an irreducible component of $\mathbf{B}_+(D),\mathbf{B}_-(D)=\mathbf{B}(D)$ can be contracted on $X$, it is covered by a family of rational curves which are contracted by the morphism contracting the component, by \cite[Theorem 1]{Kaw91}.
\end{Rem}

\begin{Rem}\label{Rem:inducedratmap}
 It is worth observing that, for an integral divisor $D$, Proposition \ref{Prop:B+MBM} is consistent with \cite[Theorem A]{BCL2014}. Indeed, adopting the same notation of Proposition \ref{Prop:B+MBM}, every $\varphi_i$ is a birational map which is an isomorphism away from its indeterminacy locus (see for example \cite[4.4-(i)]{Huybrechts1999}). Moreover $\text{Ind}(\varphi_i)=\mathrm{Exc}(\pi^{-}_i)$, for any $i=1,\dots,n-1$. Since $\codim(\text{Ind}(\varphi_i))\geq 2$ we have that
 \begin{equation}\label{eq:globalsections}
        H^0(X,D_0)\cong H^0(X,D_i)
    \end{equation}
    for any $D_i$. Moreover, as $D_n$ is big and nef, $mD_n$ is globally generated for any integer $m\gg 0$, by the base point free theorem.
    
Let $\psi$ be the morphism induced by $|mD_n|$, for $m$ large enough. Then, the map $\psi \circ \varphi$ is induced by $mD$, for any $m\gg 0$.  Let $f_{|mD|}$  be the map to some projective space induced by $|mD|$, for any $m\gg 0$. By \cite[Theorem 1]{BCL2014}, $\mathbf{B}_+(D)$ coincides with the locus where $f_{|mD|}$ is not an isomorphism. But $f_{|mD|}=\psi \circ \varphi$, because of \ref{eq:globalsections}, and in Proposition \ref{Prop:B+MBM} we proved that $\mathbf{B}_+(D)=\mathrm{MBM}(D)$, which is exactly the locus where $\psi\circ \varphi$ (and so $f_{|mD|}$) is not an isomorphism.
 \end{Rem}
 
\begin{Cor}\label{CorToThm:stablebaselocus}
Let $D$ be any big $\mathbf{R}$-divisor on $X$. Then, the divisorial part of $\mathbf{B}_{-}(D)$ consists of the prime exceptional divisors supporting the negative part of $D$, while the higher codimension irreducible components of $\mathbf{B}_{-}(D)$ are union of curves in $\mathrm{MBM}(P(D))$ which at some step of the log MMP associated with $\mathrm{MBM}(P(D))$ get flipped.
\begin{proof}
The divisorial irreducible components of $\mathbf{B}_-(D)$ are the prime exceptional divisors supporting $N(D)$ by \cite[Theorem 4.1]{Kuronya2}, for example. This shows the first part of the statement. Now, by Proposition \ref{positivepart}, we can assume that $D$ is movable (i.e.\ $N(D)=0$).  Let $f\colon X \DashedArrow X'$ be the birational map coming from the log MMP associated with $\mathrm{MBM}(D)$. If $D$ is rational, we are done, by the proof of Proposition \ref{Prop:B+MBM}. If $D$ is irrational, we can find rational, big and nef divisors $D_1',\dots,D_k'$ on $X'$, such that $f_{*}(D)=\sum_ia_iD_i'$ (with $a_i>0$ for any $i$), because $\mathrm{Nef}(X')$ is locally rational polyhedral in the big cone. If we set $D_i=f^{*}(D_i')$, we have $D=\sum_ia_iD_i$, so that $\mathbf{B}_{-}(D_i)=\mathrm{Ind}(f)$, for any $i$. Clearly, $\mathrm{Ind}(f)\subset \mathbf{B}_{-}(D)$. On the other hand $\mathbf{B}_{-}(D)\subset \cup_i \mathbf{B}_{-}(D_i)=\mathrm{Ind}(f)$, and we are done.
\end{proof}
\end{Cor}

\begin{Prop}\label{prop:picard2}
Let $D$ be any big $\mathbf{R}$-divisor on $X$, and suppose $\rho(X)=2$. Then, either all the irreducible components of $\mathbf{B}_{+}(D)$ are divisorial, or non-divisorial. Moreover, if $\mathbf{B}_{+}(D)$ is divisorial, it is irreducible.
\end{Prop}
\begin{proof}
 As $\mathbf{B}_{+}(D)=\mathbf{B}_{+}(P(D))$, we can assume that $D$ is movable. Write $D$  as $D=A-B$, where $A$ and $B$ are ample divisors on $X$. We can consider for any $\lambda \in [0,1]$ the divisor $D_{\lambda}=A-\lambda B$. If $0\leq \lambda_1\leq\lambda_2\leq 1$, $\mathbf{B}_+(D_{\lambda_1})\subset \mathbf{B}_+(D_{\lambda_2})$. Suppose $E$ is a prime divisor lying in $\mathbf{B}_+(D)$. Then, there exists $\lambda_1 \in  [0,1]$ such that $q_X(P(D_{\lambda_1}),E)=0$.
 Suppose that for some $\lambda_2\geq \lambda_1$ another component $V$ joins. Up to going on a birational model of $X$, we can assume that $V$ is covered by a family of rational curves $\{C_t\}_t$, such that $P(D_{\lambda_2})\cdot C_t=0$, for any $t$. But we also have $q_X(P(D_{\lambda_2}),E)=0$. This implies that the divisor dual to any of the curves $C_t$ is a positive multiple of $E$ because otherwise, $C_t$ would intersect negatively any ample divisor. Thus $E\cdot C_t<0$, and hence $C_t\subset E$. We conclude that $V\subset E$, thus $V=E$, because $V$ was assumed to be an irreducible component of $\mathbf{B}_+(D)$, and hence $\mathbf{B}_+(D)$ is irreducible. On the other hand, if $\mathbf{B}_+(D)$ has a non-divisorial irreducible component, all the irreducible components must be non-divisorial, because otherwise, repeating a similar argument to the one above, all the irreducible components would be contained in a divisorial irreducible component of $\mathbf{B}_+(D)$, and this would give a contradiction.
\end{proof}

\subsection{\large \itshape The asymptotic base loci are algebraically coisotropic}

For the next results, we need to introduce the notion of \emph{algebraically coisotropic subvarieties}, introduced by Voisin in \cite{Voisin16}. Let $X$ be a HK manifold of dimension $2n$ with symplectic form $\sigma \in H^{2,0}(X)$. Let $P \subset X$ be a subvariety of codimension $d$ and $P_{\mathrm{reg}}$ its regular locus. Then, the restriction of $\sigma$ to any $p \in P_{\mathrm{reg}}$ factors as
\begin{equation}
\begin{tikzcd}
    \sigma_{|P,p}: T_{P,p}\ar[r] & T_{X,p}\ar[r,"\cong"] & \Omega_{X,p}\ar[r] & \Omega_{P,p}
\end{tikzcd}
\end{equation}
where the first and last maps are given by the inclusion and restriction respectively. We say that $P$ is \emph{coisotropic} if $\sigma_{|P,p}$ has rank $2n-2d$ for every $p \in P_{\mathrm{reg}}$; if the codimension of $P$ is $n$, i.e.\ $\sigma|_{P_{\text{reg}}} = 0$, we say that $P$ is \emph{Lagrangian}. 

\begin{Def}\cite[Definition 0.5]{Voisin16} \label{defn:algebraicallycoisotropicvariety}
A subvariety $P\subset X$ of codimension $d$ is \emph{algebraically coisotropic} if it is coisotropic and admits a rational map $\phi: P\dashrightarrow B$ onto a variety of dimension $2n-2d$ such that $\sigma|_P = \phi^*\sigma_B$ for some $\sigma_B\in H^{2,0}(B)$.
\end{Def}

This last notion has received considerable attention in recent years, due to its connections with the Chow group of projective HK manifolds. Examples of algebraically coisotropic subvarieties for some of the known HK manifolds were constructed in \cite{Voisin16} (see also \cite{KLM, Lin2020}), and we will provide more. Recall that a symplectic resolution $\pi: X\to Z$ is a birational morphism from a HK manifold $X$ onto a normal variety $Z$.

\begin{Prop}\cite[Proposition 4.12]{BakkerLehn21}\label{Prop:BakkerLehn} Every irreducible component $P$ of the exceptional locus of a symplectic resolution $\pi: X\to Z$ is algebraically coisotropic and the coisotropic fibration of $P$ is given by the restriction $\pi|_P:P\to B:= \pi(P)$. Moreover, the general fiber of $\pi|_P$ is rationally connected.
\end{Prop}

From the above proposition, we deduce the following.

\begin{Cor}\label{Cor:B+MBM}
Let $D$ be a big $\R$-divisor on $X$. For any irreducible component $P$ of $\mathbf{B}_+(D)$, $\mathbf{B}(D)$ or $\mathbf{B}_{-}(D)$ there exists a birational map $\phi: X\dashrightarrow X'$ such that $\phi_*(P)$ is an algebraically coisotropic variety. In particular, $P$ has dimension at least $\mathrm{dim}(X)/2$.
\end{Cor}
\begin{proof}
By \cite[Lemma 1.14]{Ein2} and \cite[Proposition 2.8]{BBP2013} it is enough to prove the Corollary for any irreducible component of $\mathbf{B}_+(D)$. But this follows from the fact that any irreducible component of $\mathbf{B}_+(D)$ is contractible on some birational model of $X$, and from Proposition \ref{Prop:BakkerLehn}.
\end{proof}

As a refinement of the so far conjectural Bloch-Beillinson filtration in the Chow group of a HK manifold, Voisin introduced the following filtration.

\begin{Def}\cite[Definition 0.2]{Voisin16}
Let $X$ be a projective HK manifold. We define $S_iX \subset X$ to be the set of points in $X$ whose orbit under rational equivalence has dimension $\geq i$. The filtration $S_\bullet$ is then defined by letting $S_i\mathrm{CH}_0(X)$ be the subgroup of $\mathrm{CH}_0(X)$ generated by classes of points $x\in S_iX$.
\end{Def}

Following Huybrechts \cite{Huybrechts2014}, we say that a subvariety $Z\subset X$ such that all points of $Z$ are rationally equivalent in $X$ is a \emph{constant cycle subvariety}. Proposition \ref{Prop:BakkerLehn} implies that the fibers of the coisotropic fibration of $P$ are rationally connected, in particular, they are constant cycle subvarieties. The following clarifies the connection between constant cycle subvarieties and algebraically coisotropic varieties.

\begin{Thm}\cite[Theorem 1.3]{Voisin16}\label{Thm:Orbitimpliesalgebraicallycoisotropic}
Let $Z$ be a codimension $i$ subvariety of a projective HK manifold $X$. Assume that any point of $Z$ has an orbit of dimension $\geq i$ under rational equivalence in X (that is $Z \subset S_iX$). Then $Z$ is algebraically coisotropic and the fibers of the isotropic fibration are $i$-dimensional orbits of $X$ for rational equivalence.
\end{Thm}

\begin{Prop}\label{Prop:B_+isalgebraicallycoisotropic}
Let $D$ be a big $\R$-divisor on $X$. For any irreducible component $P$ of $\mathbf{B}_+(D)$, $\mathbf{B}(D)$ or $\mathbf{B}_{-}(D)$ of codimension $i$ we have that $P\subset S_iX$. In particular, $P$ is algebraically coisotropic. 
\end{Prop}
\begin{proof}
As in Corollary \ref{Cor:B+MBM} it is enough to prove the statement for an irreducible component $P\subset \mathbf{B}_+(D)$ of codimension $i$. There exists a birational map $\phi:X\dashrightarrow X'$ that induces a birational map $\phi_{|P}:P\dashrightarrow P'$. The proof of Proposition \ref{Prop:BakkerLehn} implies that $P'\subset S_iX'$. Therefore we need to prove that this is the case for $P$. Notice that the contraction map $\pi_{|P'}:P'\to B'$ is the \emph{MRC fibration} of $P'$ (see \cite[Section IV.5]{Kollar1996}, or\cite{GHS2003}, or \cite[Section 3]{Voisin2021} for properties of the MRC fibration). Therefore, by restricting to a possible smaller open subset we can assume that the rational map $\phi: P\dashrightarrow B'$ induced by composition is the MRC fibration of $P$ as well. In particular, since $\phi$ is almost holomorphic, we can assume that there exists an open subset $U\subset P$ such that $\phi|_U: U \to B'$ is flat and the fibers are rationally connected subvarieties of $P$. Fix an ample class on $P$, then the fibers of the map  $\phi|_U: U\to B'$ have the same Hilbert polynomial with respect to the fixed ample class. The relative Hilbert scheme (or the Kontsevich space of stable maps) is proper \cite{Kollar1996}, hence for every point $p\in P$ there exists a limit of the rationally connected varieties $\phi^{-1}(b)$ containing $p$. Since the limits of rationally connected varieties are rationally chain connected, we obtain that the orbit under rational equivalence in $P$ is of dimension $\geq i$, therefore $P\subset S_iX$. We conclude that the variety is algebraically coisotropic by Theorem \ref{Thm:Orbitimpliesalgebraicallycoisotropic}.
\end{proof}

We conclude this section by proving the first three items of Theorem \ref{Thm:baselociHK}.

\begin{proof1}
 The irreducible components of $\mathbf{B}_+(D), \mathbf{B}_-(D)=\mathbf{B}(D)$ are algebraically coisotropic  by Proposition \ref{Prop:B_+isalgebraicallycoisotropic}. The bound from below for the dimension of any such irreducible component follows from the definition of algebraically coisotropic subvariety. The statement about the divisorial irreducible components of $\mathbf{B}_+(D)$ (resp.\ $\mathbf{B}_-(D)$) is proven in Proposition \ref{lem:B+positivepart} (resp.\ Corollary \ref{CorToThm:stablebaselocus}). The rest directly follows from Proposition \ref{Prop:B+MBM} and Corollary \ref{CorToThm:stablebaselocus}. This concludes the proof of the first three items of Theorem \ref{Thm:baselociHK}.
\end{proof1}

\section{\bf \scshape Stable Classes and Destabilizing Numbers}\label{Section3}
In this section, we describe the parts of the big cone of a projective HK manifold where the augmented base loci stay constant. 

\begin{Rem}
By \cite[Theorem A]{BBP2013}, a big divisor $D$ on $X$ is unstable if and only if $\mathbf{B}(D) \subsetneq \mathbf{B}_{+}(D)$.
\end{Rem}

\begin{Prop}\label{Prop:unstabledivisors}
 A big, movable divisor $D$ on $X$ is unstable if and only if there exist a birational hyper-Kähler model $X'$ of $X$, a birational map $f\colon X \DashedArrow X'$, a rational curve $C\in \mathrm{Cont}(X)$, with $C \not\subset \mathrm{Ind}(f)$, such that $f_{*}(D)$ is nef on $X'$, and $f_{*}(D) \cdot f_{*}(C)=0$.
\end{Prop}
\begin{proof} 
If the divisor $D$ is unstable, then $\mathbf{B}(D)\subsetneq \mathbf{B}_+(D)$. Let $p$ be a point lying in $\mathbf{B}_+(D)$ but not in $\mathbf{B}(D)$, and consider a log MMP 
\begin{equation}\label{lastdiagram}
            \begin{tikzcd}
                X=X_1 \arrow[r, dashed,"\varphi_1"] \arrow[rrr, "f", bend left = 35]& X_2 \arrow[r, dashed,"\varphi_2"]  & \cdots \arrow[r, dashed, "\varphi_{k-1}"]  & X_k            
               \end{tikzcd}
\end{equation}
for the pair $(X,\epsilon D)$. We know that $\mathbf{B}(D)=\mathrm{Ind}(f)$, hence $p$ does not belong to the indeterminacy locus of $f$. By Proposition \ref{Prop:B+MBM} and Corollary \ref{CorToThm:stablebaselocus}, there exists a rational curve $C$ passing through $p$, belonging to $\mathrm{MBM}(D)$, which survives until the end of the log MMP we have chosen. In particular, we must have $f_*(D)\cdot f_*(C)=0$, because, since $C$ lies in $\mathrm{MBM}(D)$, it must be contracted by the morphism induced by a certain line bundle, whose class lies in the relative interior of the minimal extremal face of the nef cone of $X_k$ containing $f_*(D)$.
Now, suppose  $f\colon X \DashedArrow X'$ is a birational map of HK manifolds and $C\in \mathrm{Cont}(X)$ a rational curve with $C \not\subset \mathrm{Ind}(f)$, such that $f_{*}(D)$ is nef on $X'$, and $f_{*}(D) \cdot f_{*}(C)=0$, we want to show that $\mathbf{B}(D) \subsetneq \mathbf{B}_+(D)$, i.e.\ that $D$ is unstable. The birational map $f$ is a composition of log flips. For simplicity, let us keep the notation of diagram (\ref{lastdiagram}). Note that $\varphi_{i-1,*}(C)$ is not contained in the indeterminacy locus of $\varphi_{i}$, for any $i=1,\dots,k-1$, and we set $\varphi_{0,*}(C)=C$. Again, the curve $\varphi_{k-1,*}(C)$ is contracted by the morphism induced by a certain line bundle, whose class lies in the relative interior of the minimal extremal face of the nef cone of $X_k$ containing $f_*(D)$, because $f_*(D)\cdot f_*(C)=0$. This implies that $C \in \mathrm{MBM}(D)$ (here $\mathrm{MBM}(D)$ is taken with respect to the log MMP in diagram (\ref{lastdiagram})), and so, by Proposition \ref{Prop:B+MBM}, we have that $C \subset \mathbf{B}_+(D)$ and $C \not\subset \mathbf{B}(D)=\mathrm{Ind}(f)$, hence $D$ is unstable.
\end{proof}

\begin{Prop}\label{Prop:destabnumbers}
Let $\lambda$ be a destabilizing number for a big, integral divisor $D$ on $X$, with respect to an ample, integral divisor $A$. If $D-\lambda A$ is big, $\lambda$ is rational.
\begin{proof}
We know that $D-\lambda A$ is unstable if and only if $\mathbf{B}(D-\lambda A) \subsetneq \mathbf{B}_+(D-\lambda A)$. In particular, up to going on a birational model of $X$, if $D-\lambda A$  is unstable, there exists a rational curve $C$, such that $P(D-\lambda A)\cdot C=0$, by Proposition \ref{Prop:unstabledivisors}. Then, $\lambda$ is forced to be a rational number, and this concludes the proof.
\end{proof}
\end{Prop}
We point out that, in general, the destabilizing numbers are not rational numbers. An example of the irrationality of such numbers can be found in \cite{Bau}, as well as the proof that the destabilizing numbers on smooth projective surfaces are rational.
 \begin{Rem}
Note that, in the above proposition, the assumption on $D-\lambda A$ to be big cannot be dropped. Indeed, if $\lambda$ is a destabilizing number for $D-\lambda A$, then $D-\lambda A$ is pseudo-effective but not big, and lying on an irrational ray of $\overline{\mathrm{Eff}(X)}$, $\lambda$ must be irrational (because otherwise, $D-\lambda A$ would be rational).
 \end{Rem}

We now put all the above together to prove the item (4) of Theorem \ref{Thm:baselociHK}.

\begin{proof2}
 The part concerning the unstable divisors is proven in Proposition \ref{Prop:unstabledivisors}, while the part concerning the destabilizing numbers is proven in Proposition \ref{Prop:destabnumbers}.
\end{proof2}

The rest of this section is devoted to study how the asymptotic base loci of a big divisor vary when perturbing the divisor.

\begin{Def}\label{defnstabchamb}
Let $D$ be a stable big $\mathbf{R}$-divisor on $X$. The stability chamber of $D$ is defined as
\[
\mathrm{SC}(D):=\left\{D' \in \mathrm{Big}(X) \; | \; \mathbf{B}_{+}(D')=\mathbf{B}_{+}(D)\right\}.
\]
\end{Def}
Notice that by \cite[Proposition 1.26]{Ein2} and \cite[Proposition 1.24, item (iii)]{Ein2}, it follows that any stability chamber in $\mathrm{Big}(X)$ has a non-empty interior, thus, being stable is an open condition.
\vspace{0.2cm}

We notice that, in general, the stability chambers are not convex subcones of the big cone. An example of this pathology is provided in \cite[Example 3.1]{LMR2020} for Mori dream spaces. We provide an example of this pathology on a projective HK manifold of $\text{K3}^{[2]}$-type which is not a Mori dream space (see the subsection "The example of Hassett and Tschinkel"). This cannot happen on surfaces (hence in particular on K3 surfaces), because in that case, the augmented base loci stay constant in the interior of the Zariski chambers (see \cite[Section 1]{Bau} for the definition of Zariski chamber, and \cite[Theorem 2.2]{Bau} for the cited result), which are convex subcones of the big cone.

\begin{Def}
We say that a big divisor $D$ on $X$ is stable in codimension 1 if $\mathbf{B}_+(D)_{\mathrm{div}}=\mathbf{B}_{-}(D)_{\mathrm{div}}$.
\end{Def}

The result below, which is a corollary to Proposition \ref{lem:B+positivepart}, proves that the divisorial augmented base locus of big divisors on a projective HK manifold $X$ stay constant in the interior of any Boucksom-Zariski chamber (cf.\ \cite[Definition 4.11]{Den1}) of $\mathrm{Big}(X)$. 

\begin{Cor}\label{cor:stableincodim1}
Let $D$ be a big divisor on $X$. Then, $D$ is stable in codimension 1 if and only if $\mathrm{Null}_{q_X}(P(D))=\mathrm{Neg}_{q_X}(D)$, if and only if (the class of) $D$ does not lie on the boundary of any Boucksom-Zariski chamber, where $\mathrm{Neg}_{q_X}(D)$ is the set of the irreducible components of $N(D)$.
\end{Cor}
\begin{proof}
This directly follows from Proposition \ref{lem:B+positivepart} and \cite[Proposition 4.18]{Den1}.
\end{proof}

\begin{Cor}{}{}
Let $D$ be any big divisor on $X$. Then, $D$ is unstable if and only if $\mathrm{Neg}_{q_X}(D) \subsetneq \mathrm{Null}_{q_X}(P(D))$, or $P(D)$ is unstable in codimension greater than or equal to 2 (i.e.\ $\mathbf{B}_+(P(D))$ has an irreducible component $V$ of codimension greater than or equal to $2$, with $V \not \subset \mathbf{B}_-(P(D))$).
\end{Cor}
\begin{proof}
This directly follows from Corollary \ref{cor:stableincodim1} and from the fact that $\mathbf{B}_+(D)=\mathbf{B}_+(P(D))$. 
\end{proof}

\begin{Lem}\label{lem:convexity}
Suppose that any big class $\gamma \in \overline{\mathrm{Mov}(X)}$ on a projective HK manifold $X$ satisfies $\gamma \cdot C\leq 0$, for any curve of $\mathrm{Cont}(X)$ contained in $\mathbf{B}_+(\gamma)$. Then, $\mathrm{SC}(D)$ is a convex subcone of $\mathrm{Big}(X)$ for any big divisor $D$.
\begin{proof}
Let $\alpha,\beta$ be two classes lying in $\mathrm{SC}(D)$, so that $\mathbf{B}_+(\alpha)=\mathbf{B}_+(\beta)$. We want to prove that $\alpha+\beta \in \mathrm{SC}(D)$. One easily checks that  $\mathbf{B}_+(\alpha'+\beta')\subset \mathbf{B}_+(\alpha') \cup \mathbf{B}_+(\beta')$ (cf.\ \cite[Example 1.9]{Ein2}), for any $\alpha',\beta' \in  N^1(X)_{\mathbf{R}}$, and so $\mathbf{B}_+(\alpha+\beta)\subset \mathbf{B}_+(\alpha)$. Suppose by contradiction that $p \in \mathbf{B}_+(\alpha) $ but $p \notin \mathbf{B}_+(\alpha+\beta)$. Let $C \in \mathrm{Cont}(X)$ be a rational curve contained in $\mathbf{B}_+(\alpha)=\mathbf{B}_+(\beta)$ passing through $p$. Then $C \not\subset \mathbf{B}_+(\alpha+\beta)$, and so $(P(\alpha+\beta))\cdot C=(P(\alpha)+P(\beta))\cdot C >0$, whereas $P(\alpha+\beta)\cdot C \leq 0$, by assumption, and this is a contradiction. We conclude that $\mathbf{B}_+(\alpha+\beta)= \mathbf{B}_+(\alpha)=\mathbf{B}_+(\beta)$.

\end{proof}
\end{Lem}

The above lemma tells us that, on HK manifolds, to have a non-convex stability chamber, there must be some curve $C$ in $\mathrm{Cont}(X)$, and some movable divisor class $\gamma$, such that $C \subset \mathbf{B}_+(\gamma)$, and $\gamma \cdot C>0$. This happens for instance for the Fano variety of lines of certain smooth cubic fourfolds (see "The example of Hassett and Tschinkel").

\section{\bf \scshape Duality for Cones of $k$-Ample Divisors}\label{duality}
In this section, we prove Theorem \ref{Thm:duality}.

\begin{Def}
Let $\overline{\mathrm{Amp}_k(X)}$ be the closure of the convex cone $\mathrm{Amp}_k(X)$ spanned by divisor classes $\gamma$ with $\mathrm{dim}(\mathbf{B}_{+}(\gamma))\leq k-1$. We say that a class lying in $\mathrm{Amp}_k(X)$ is $k$-ample.
\end{Def}

\begin{Rem}
Note that the terminology introduced in the above definition is not standard. Indeed, another notion of "$k$-ampleness" is introduced by Totaro in \cite{Totaro13} and has been intensively studied.
\end{Rem}

\begin{Ex}
We have $\overline{\mathrm{Amp}_{2n-1}(X)}=\overline{\mathrm{Mov}(X)}$. Indeed, we have the equality $\mathrm{Amp}_{2n-1}(X)=\mathrm{Int}(\mathrm{Mov}(X))$ of convex cones. 
\end{Ex}

Let $\phi:X\dashrightarrow X'$ be a birational map to a projective HK manifold $X'$. Since  $N^1(X)_{\mathbf{R}}$ and $N^1(X')_{\mathbf{R}}$ are isomorphic under $\phi_*$, their dual spaces $N_1(X)_{\mathbf{R}}$ and $N_1(X')_{\mathbf{R}}$ are also isomorphic. Under this isomorphism, any class $\alpha \in N_1(X')_{\mathbf{R}}$ can be pulled-back to 
a class in $N_1(X)_{\mathbf{R}}$. We define
\begin{equation}
   \overline{\mathrm{Mob}_k(X,X')}\subset N_1(X)_{\mathbf{R}} 
\end{equation}
to be the image of the convex cone generated by numerical classes of irreducible curves $C$ in $X'$ moving in a family that sweeps out the birational image of a subvariety of $X$ of dimension at least $k$, via the isomorphism $N_1(X)_{\mathbf{R}}\cong N_1(X')_{\mathbf{R}}$. Now, consider the cone  
\[\
\overline{\sum_{X\dashrightarrow X'}\overline{\mathrm{Mob}_k(X,X')}},
\]
where the sum is taken over all birational maps $X\dashrightarrow X'$, such that $X'$ is a projective HK manifold.

\begin{Def}
 We define the cone of birationally $k$-mobile curves of $X$  as \[
 \mathrm{bMob}_k(X):=
\overline{\sum_{X\dashrightarrow X'}\overline{\mathrm{Mob}_k(X,X')}}.
\]
\end{Def}

We now prove Theorem \ref{Thm:duality}.

\begin{proof3}
The inclusion $\subset$ is clear. Hence, we are left to show $\supset$. We first observe that 

\[
\mathrm{bMob}_k(X)^{\vee} \subset \overline{\mathrm{Eff}(X)}.
\]

Indeed, $\overline{\mathrm{Mob}(X)}\subset \mathrm{bMob}_k(X)$, thus, passing to the duals, we obtain
$\mathrm{bMob}_k(X)^{\vee}\subset \overline{\mathrm{Eff}(X)}$ (cf.\ \cite[Theorem 2.2]{BDPP}). Now, we show that $\mathrm{bMob}_{2n-1}(X)^{\vee}\subset \overline{\mathrm{Mov}(X)}$. Indeed, if $E$ is a prime exceptional divisor that can be contracted via $\pi \colon X \to Y$ (with $Y$ normal and projective), the class of a general fiber of $\pi$ is given by $-2\frac{q_X(E,-)}{q_X(E)}$ (cf.\ \cite[Corollary 3.6, part 1]{Markman}). If $E$ is not contractible, it can be contracted on a birational model $f \colon X \DashedArrow X'$ (cf.\ \cite[Proposition 1.4]{Druel}, or \cite[Proof of Theorem A]{BBP2013}), and, on $X'$, either the homology class $-2\frac{q_X(E',-)}{q_X(E')}$, or the homology class $-\frac{q_X(E',-)}{q_X(E')}$, is represented by a rational curve (cf.\ \cite[Corollary 3.6, part 1]{Markman}) moving in a family sweeping out the strict transform $E'\subset X'$ of $E$ via $f$. Now, $\overline{\mathrm{Mov}(X)}$ consists of the pseudo-effective classes $q_X$-intersecting non-negatively any prime exceptional divisor. Indeed, any class in $\overline{\mathrm{Mov}(X)}$ is $q_X$-nef, i.e.\ $q_X$-intersect non-negatively any prime divisor. On the other hand, suppose that $\alpha$ is a pseudo-effective class $q_X$-intersecting non-negatively any prime exceptional divisor. If $\alpha \notin \overline{\mathrm{Mov}(X)}$, we have $\alpha=P(\alpha)+N(\alpha)$, with $N(\alpha)\neq 0$. Then, since $P(\alpha)\in \overline{\mathrm{Mov}(X)}$ and $q_X$ is an intersection product (i.e.\ $q_X(E,E')\geq 0$ for any two distinct prime divisors $E,E'$ on $X$), there must be some prime exceptional divisor supporting $N(\alpha)$, and $q_X$-intersecting negatively $\alpha$. But this contradicts the assumption on $\alpha$. Then, $\alpha=P(\alpha) \in \overline{\mathrm{Mov}(X)}$, and we obtain the desired inclusion. Then, the interior of $\mathrm{bMob}_k(X)^{\vee}$ is contained in $\mathrm{Big}(X)$, and to prove the theorem it suffices to show that 
\[
\mathrm{int}\left(\mathrm{bMob}_k(X)^{\vee}\right) \subset \mathrm{Amp}_k(X).
\]
Suppose that $D$ is a big divisor with class lying in  $\mathrm{int}\left(\mathrm{bMob}_k(X)^{\vee}\right)$, but not in $\mathrm{Amp}_k(X)$.  Then, there exists an irreducible component $V$ of $\mathbf{B}_+(D)$ with $\mathrm{dim}(V)\geq k$. As $D\in \mathrm{Int}(\mathrm{Mov}(X))$, by \cite[Proof of Theorem A]{BBP2013} and the proof of Proposition \ref{Prop:B+MBM}, there exists a class $\alpha$ and a birational map $\phi:X\dashrightarrow X'$, with $X'$ a projective HK manifold, inducing a birational map $\phi|_V:V\dashrightarrow V'$ (where $V'$ is the strict transform of $V$ via $\phi$), and such that $\phi_*(\alpha)$ is the class of a curve which moves in a family sweeping out $V'$, and with $\phi_*(D)\cdot \phi_*(\alpha) \leq 0$. But 
 then $D\cdot \alpha \leq 0$, and 
 $\alpha \in \mathrm{bMob}_k(X)$. This is a contradiction, because $D$ intersects positively any class lying in $\mathrm{bMob}_k(X)$, as $D\in \mathrm{int}\left(\mathrm{bMob}_k(X)^{\vee}\right)$. The second part of the statement directly follows from Corollary \ref{Cor:B+MBM}.
\end{proof3}

\section{\bf \scshape A Decomposition of $\mathrm{Eff}(X)$ into Chambers of Mori-Type }\label{Section6}
It is well known that the birational geometry of a Mori dream space is encoded by its effective cone, and, more in particular, by the decomposition of its effective cone into Mori chambers.
On hyper-Kähler manifolds, the minimal model program works nicely (cf.\ \cite[Theorem 4.1]{MatsushitaZhang2013}, \cite[Theorem 1.2]{LP2016}), and, as already remarked, projective HK manifolds are close to being Mori dream spaces. For this reason, also on projective HK manifolds, one may suspect the existence of a decomposition of the effective cone into chambers of "Mori-type", resuming the birational geometry of the considered variety. In this section, we show that we indeed have such a decomposition for $\mathrm{Eff}(X)$, which we still call decomposition into \textit{Mori chambers}. 
\vspace{0.2cm}

Fix $Y$ to be a normal complex $\mathbf{Q}$-factorial projective variety.

\begin{Def}
 Let $f\colon Y \DashedArrow Z$ be a dominant rational map, with $Z$ normal and projective. We say that $f$ is a rational contraction if there exists a resolution of $f$
 \[
 \begin{tikzcd}
  Y'  \arrow[d, "\mu"] \arrow[dr, "f'"]
& \\
Y \arrow[r,dashrightarrow, "f"]
& Z, 
 \end{tikzcd}
 \]
 where $Y'$ is smooth and projective, $\mu$ is birational, and for every $\mu$-exceptional effective divisor $E'$ on $Y'$, we have $f'_*(\mathscr{O}_{Y'}(E'))=\mathscr{O}_Z$. 
\end{Def}

\begin{Rem}
Let $Y'$ be a projective $\mathbf{Q}$-factorial variety. If $ s \colon Y' \DashedArrow Y$ is a small $\mathbf{Q}$-factorial modification and $c \colon Y \DashedArrow Z$ is a rational contraction, with $Z$ normal and projective, also $c \circ s$ is a rational contraction.
\end{Rem}

If in the above definition, $f$ is a morphism, $f$ is a contraction in the usual sense.
\vspace{0.2cm}

Now, let $D$ be an effective divisor on $X$. Then, either $D \in \overline{\mathrm{Mov}(X)}$, so that $D\in f^{-1}(\mathrm{Nef}(X'))$, for some projective hyper-Kähler manifold $X'$, and some birational map $f \colon X \DashedArrow X'$, or $D$ has some divisorial stable base locus. In the second case, let $D=P+N$ be the divisorial Zariski decomposition of $D$, and $f \colon X \DashedArrow X'$ be any birational map making $P$ nef on $X'$ (projective HK). We consider the face of $\overline{\mathrm{Mov}(X)}$
\[
F=\mathrm{Neg}_{q_X}(D)^{\perp}\cap \overline{\mathrm{Mov}(X)},
\]
where $\mathrm{Neg}_{q_X}(D)$ is the set of irreducible components of $N$, and
\[
\mathrm{Neg}_{q_X}(D)^{\perp}:=\left\{\alpha\in N^1(X)_{\mathbf{R}}\;|\; q_X(\alpha,E)=0 \text{ for any prime divisor $E$ supporting }N\right\}.
\] 

Let $c\colon X' \to Y'$ be the birational morphism contracting the face $F'^{\vee}\subset \overline{\mathrm{NE}(X')}$ (where $P'=f_*(P)$, and $F'=f_*(F)$) to the normal $\mathbf{Q}$-factorial symplectic variety $Y'$ (we can do this because the Gram-matrix of the irreducible components of $N'=f_*(N)$ is negative definite, if $N'\neq 0$). In particular, we have $c^{-1}(\mathrm{Nef}(Y'))=F'$, and so $(c \circ f)^{-1}(\mathrm{Nef}(Y'))=F$. Then, $D\in F \cap \mathrm{Eff}(X)+V^{\geq 0}(\mathrm{Neg}_{q_X}(D))$, where $V^{\geq 0}(\mathrm{Neg}_{q_X}(D))$ is the conic convex hull of $\mathrm{Neg}_{q_X}(D)$. Vice versa, let $f\colon X \DashedArrow Y'$ be a birational contraction to a normal, $\mathbf{Q}$-factorial projective variety. Then, $f$ factors through a birational map $g\colon X \DashedArrow X'$, with $X'$ a projective HK manifold, and a surjective birational morphism $c\colon X' \to Y'$ with connected fibers. As $Y'$ is $\mathbf{Q}$-factorial, the irreducible components of the exceptional locus $\mathrm{Exc}(c)$ are divisorial. Let $S=\left\{E'_1,\dots,E'_k\right\}$ be the set of irreducible components of $\mathrm{Exc}(c)$. The Gram matrix of $S$ must be negative definite (if $S$ is not empty), and $c^{-1}(\mathrm{Nef}(Y'))=S^{\perp}\cap \mathrm{Nef}(X')$. Set
$F=(c\circ f)^*(\mathrm{Nef}(Y'))$. Then, any divisor of the form $D=A+B$, where $A\in F\cap \mathrm{Eff}(X) $, and $B=\sum_ia_if^*(E_i')$, $a_i\geq 0$ for every $i$, is effective. Furthermore, $A$ (resp.\ $B$) is the positive (resp.\ negative) part of the divisorial Zariski decomposition of $D$. We conclude that the class of any effective divisor $D$ on $X$ lies in the subcone of $\mathrm{Eff}(X)$ 
\begin{equation}\label{Morichambers}
F \cap \mathrm{Eff}(X)+V^{\geq 0}(\mathrm{Neg}_{q_X}(D)),
\end{equation}
where $F=\mathrm{Neg}_{q_X}(D)^{\perp}\cap \overline{\mathrm{Mov}(X)}$, and induces a birational contraction to a normal $\mathbf{Q}$-factorial projective variety. Vice versa, any birational contraction to a normal $\mathbf{Q}$-factorial projective variety induces a subcone of $\mathrm{Eff}(X)$ of the form (\ref{Morichambers}). We call the subcones of $\mathrm{Eff}(X)$ of the form (\ref{Morichambers}) \textit{Mori chambers} of $\mathrm{Eff}(X)$. Note that two different Mori chambers can intersect only on the boundary, i.e.\ the interiors of two different Mori chambers are disjoint. We resume all the above with the following proposition. 

\begin{Prop}
 Let $X$ be a projective HK manifold. Then
 \[
 \mathrm{Eff}(X)=\cup_i\mathcal{C}_i,
 \]
 where \[
 \mathcal{C}_i=g_i^{*}(\mathrm{Nef}(Y_i))\cap \mathrm{Eff}(X)+\sum_j\mathbf{R}^{\geq 0}E_i^j,
 \]
 $g_i\colon X \DashedArrow Y_i$ is a birational contraction, with $Y_i$ projective, normal and $\mathbf{Q}$-factorial, and the $E_i^j$'s are the prime exceptional divisors contracted by $g_i$, for any $i$.
\end{Prop}

The decomposition of $\mathrm{Eff}(X)$ into Mori chambers induces a decomposition of $\mathrm{Big}(X)$ into convex subcones, which we still call decomposition of $\mathrm{Big}(X)$ into \textit{Mori chambers}. In the interior of any Mori chamber of $\mathrm{Big}(X)$, the stable base locus of the divisor classes is constant. However, a locus of the big cone where the stable base locus is constant could be bigger than the interior of a Mori chamber  (see Remark \ref{Morifinerthanstabchambers} for an example). It is worth observing that the decomposition of $\mathrm{Big}(X)$ into Boucksom-Zariski chambers (cf.\ \cite{Den1}) is related to that in Mori chambers. Indeed, the interior of any Boucksom-Zariski chamber different from $\overline{\mathrm{Mov}(X)}\cap \mathrm{Big}(X)$ is equal to the interior of a unique Mori chamber (and vice versa). The point is that the Boucksom-Zariski chambers should be thought of as "stability chambers in codimension 1". In particular, in the interior of each Boucksom-Zariski chamber, the divisorial augmented base locus is constant, and a big divisor $D$ lies on the boundary of some Boucksom-Zariski chamber if and only if $D$ is unstable in codimension one, i.e.\ if and only if $\mathbf{B}_-(D)_{\mathrm{div}}\subsetneq \mathbf{B}_+(D)_{\mathrm{div}}$ (see Corollary \ref{cor:stableincodim1}). This is coherent with the fact that the big classes in $\mathrm{int}(\mathrm{Mov}(X))$ have an empty divisorial augmented base locus, though the big classes in $\mathrm{Mov}(X)$ could have some non-divisorial augmented base locus.
\section{\bf \scshape Examples}\label{Examples}

In this section, we provide several examples. 

\subsection{\large \itshape Hilbert schemes of points on K3 surfaces}\label{Example:Hilbertschemes}
Let $S_d$ be a K3 surface such that $\Pic(S_d) \cong \Z\cdot H_{S_d}$, with $H_{S_d}^2 = 2d$. Consider the projective HK manifold $\mathrm{Hilb}^2(S_d)$ with the usual decomposition $\mathrm{Pic}(\mathrm{Hilb}^2(S_d))\cong \mathbf{Z} \cdot H \oplus \mathbf{Z} \cdot \delta$, where $H$ is the line bundle associated to $H_{S_d}$ and $\delta$ is half the divisor corresponding to non-reduced subschemes or the exceptional divisor of the Hilbert-Chow morphism $\mathrm{Hilb}^2(S_d)\to \mathrm{Sym}^2(S_d)$. The effective cone of $\mathrm{Hilb}^2(S_d)$ and its decomposition into birational chambers is completely determined by the solution of some Pell equations due to the work of Bayer-Macrì, see \cite[Lemma 13.1, Proposition 13.3]{BayerMacri2014}. For polarized K3 surfaces $S_d$ of low degrees one can give a precise description of the stability chambers of $\mathrm{Big}(\mathrm{Hilb}^2(S_d))$.

\begin{Ex}\label{Ex:SCDdegree2}
Assume that $d=1$. Then $S_1$ is a double covering of $\mathbf{P}^2$ ramified along a sextic curve. Let $\pi:S_1 \to \mathbf{P}^2$ be such covering. We have:
\begin{enumerate}
\item $\mathrm{Eff}(\mathrm{Hilb}^2(S_1))=\langle H-\delta,\delta\rangle$,
\item $\mathrm{Mov}(\mathrm{Hilb}^2(S_1))=\langle H-\delta, H\rangle$,
\item $\mathrm{Nef}(\mathrm{Hilb}^2(S_1))=\langle 3H-2\delta, H\rangle$.
\end{enumerate}

We see that there are three stability chambers: $\mathrm{SC}(A)$, which is associated with any ample divisor $A$, the chamber $\mathrm{SC}(H)$, and $\mathrm{SC}(3H-2\delta)$. We can describe all of them geometrically:
\begin{enumerate}
\item In $\mathrm{SC}(A)$ the augmented base locus is empty.
\item In $\mathrm{SC}(H)$ the augmented base locus is given by the exceptional divisor $E \equiv 2\delta$. The divisor $H$ is big and nef, and the morphism induced by $|H|$ is nothing but the Hilbert-Chow morphism $\mathrm{Hilb}^2(S_d)\to \mathrm{Sym}^2(S_d)$ (whose exceptional locus is given by $E$).
\item In $\mathrm{SC}(3H-2\delta)$ the augmented base locus is given by the plane $\mathbf{P}^2$, obtained using the fibers of the double covering $\pi$. The other birational model of $\mathrm{Hilb}^2(S_1)$ is obtained by taking the \emph{Mukai flop} along $\mathbf{P}^2$, i.e.\ the blow-up of $\mathbf{P}^2$ in $\mathrm{Hilb}^2(S_1)$, followed by the contraction of the exceptional divisor in the other direction.
\end{enumerate} 

\begin{figure}[htbp]
  \centering

  \begin{tikzpicture}[scale=0.9]

       \path[fill=Ivory1] (0,0) -- (0,5) -- (4,5) -- (4,0) -- cycle;
       \path[fill=lightgray] (0,0) -- (0,5) -- (-3.33,5) -- cycle;
        \path[fill=PeachPuff2] (0,0) -- (-3.33,5) -- (-4,5) -- (-4,4)  --cycle;

\draw [Ivory1](0,0) -- (0,5);
\draw [line width=1pt] (0,0) -- (-4,4);
\draw [line width=1pt]  (0,0) -- (4,0);
\draw [PeachPuff2] (0,0) -- (-3.33,5);

       \node (nefX) at (-1.3,4.5) {$\mathrm{SC}(A)$};
       \node (nefX') at (-4.4,4.5) {$\mathrm{SC}(3H-2\delta)$};
       \node (sigmaH) at (2.3,2.3) {$\mathrm{SC}(H)$};

     \foreach \x in {-4,-3,...,4} 
       \foreach \y in {-1,...,5}{
         \fill(\x,\y) circle (1pt);}

       \fill(0,0) circle (2 pt);
       \node (o) at (-0.3,-0.3) {$0$};
       \fill(1,0) circle (2 pt);
       \node (d) at (1,-0.3) {$\delta$};
       \fill(0,1) circle (2 pt);
       \node (h) at (0.3,1) {$H$};
       \fill(-1,1) circle (2 pt);
       \node (l) at (-1.4,0.5) {$H-\delta$};

  \end{tikzpicture}
  \caption{Stability chambers on $\mathrm{Hilb}^2(S_1)$}
  \label{figure:SCDdegree2}
\end{figure}
\end{Ex}

\begin{Ex}
With the same notation as above, let $d =2$. Then  $S_2\subset\P^3$ is a quartic surface, and we get the following decomposition of the effective cone:
\begin{enumerate}
\item $\mathrm{Eff}(\mathrm{Hilb}^2(S_2))=\langle 2H-3\delta,\delta\rangle$,
\item $\mathrm{Mov}(\mathrm{Hilb}^2(S_2))=\Nef(\mathrm{Hilb}^2(S_1)) =\langle 3H-4\delta, H\rangle$.
\end{enumerate}
Notice that in this case, we have an involution $\iota:\mathrm{Hilb}^2(S_2)\to \mathrm{Hilb}^2(S_2)$ given by the residual intersection with the line spanned by two different points. Let $E=2\delta$ be the exceptional divisor. Then, the involution $\iota$ does not fix $E$, because a tangent line that is not bi-tangent maps a non-reduced double point to two points. Moreover, one can check that the class of $\iota(E)$ in the Néron-Severi space is  $\iota^*(E) = 2H-3\delta$, and this is represented by a prime exceptional divisor different from $E$. All the asymptotic base loci of big divisors are divisorial in this case, and there are three stability chambers: $\mathrm{SC}(A)$ where $A$ is any ample divisor on $\mathrm{Hilb}^2(S_2)$, the chamber $\mathrm{SC}(H)$, and $\mathrm{SC}(3H-4\delta)$, see Figure \ref{fig:SCDdegree4}. The first two chambers are as in Example \ref{Ex:SCDdegree2}. In the third, the augmented base locus is given by $i(E)$, and we have $\mathrm{SC}(2H-3\delta)=\langle 2H-3\delta,3H-4\delta\rangle \cap \mathrm{Big}(\mathrm{Hilb}^2(S_2))$.

\begin{figure}[htbp]
  \centering

  \begin{tikzpicture}[scale=0.9]

       \path[fill=Ivory1] (0,0) -- (0,5) -- (4,5) -- (4,0) -- cycle;
       \path[fill=lightgray] (0,0) -- (-4,2.66) -- (-4,3) -- cycle;
        \path[fill=PeachPuff2] (0,0) -- (-4,3) -- (-4,5) -- (0,5)  --cycle;

\draw [Ivory1](0,0) -- (0,5);
\draw [line width=1pt] (0,0) -- (-4,2.66);
\draw [line width=1pt]  (0,0) -- (4,0);
\draw [PeachPuff2] (0,0) -- (-3.33,5);

\node (pt1) at (-3,0.4){};

\draw [->] (pt1) to[bend left] (-3.6,2.6);

       \node (nefX) at (-2,3.3) {$\mathrm{SC}(A)$};
       \node (nefX') at (-2.5,0.3) {$\mathrm{SC}(3H-4\delta)$};
       \node (sigmaH) at (2.3,2.3) {$\mathrm{SC}(H)$};

     \foreach \x in {-4,-3,...,4} 
       \foreach \y in {-1,...,5}{
         \fill(\x,\y) circle (1pt);}

       \fill(0,0) circle (2 pt);
       \node (o) at (-0.3,-0.3) {$0$};
       \fill(1,0) circle (2 pt);
       \node (d) at (1,-0.3) {$\delta$};
       \fill(0,1) circle (2 pt);
       \node (h) at (0.3,1) {$H$};
           \fill(-3,2) circle (2 pt);
       \node (e) at (-2.3,2.3) {$2H-3\delta$};

  \end{tikzpicture}
  \caption{Stability chambers on $\mathrm{Hilb}^2(S_2)$}
  \label{fig:SCDdegree4}
\end{figure}

\end{Ex}

\begin{Ex}
Let $d=3$, then $S_3\subset \P^4$ is the complete intersection of a quadric $Q\subset\P^4$ and a cubic $T\subset \P^4$. The decomposition of the effective cone is:
\begin{enumerate}
\item $\mathrm{Eff}(\mathrm{Hilb}^2(S_3))=\langle H-2\delta,\delta\rangle$,
\item $\mathrm{Mov}(\mathrm{Hilb}^2(S_3))=\Nef(\mathrm{Hilb}^2(S_3)) =\langle 2H-3\delta, H\rangle$.
\end{enumerate}
This case also corresponds to divisorial stability chambers. As before we have three stability chambers: $\mathrm{SC}(A)$ where $A$ is an ample divisor on $\mathrm{Hilb}^2(S_3)$, the chamber $\mathrm{SC}(H)$ and $\mathrm{SC}(3H-4\delta)$. In the third stability chamber, the augmented base locus is given by an irreducible divisor $D$. One can describe $D$ geometrically (see \cite[Lemma 4.5]{GounelasOttem2020}) as follows: the quadric $Q$ is a hyperplane section of $\mathrm{Gr}(2,4)\subset\P^5$ under the Plücker embedding, then $D\cong \P(\mathcal{U}^\vee|_S)$, where $\mathcal{U}$ is the tautological rank 2 bundle on $\mathrm{Gr}(2,4)$. A computation of its Chern classes yields that $D$ is not isomorphic to $E$.
\end{Ex}

\subsection{\large \itshape The example of Hassett and Tschinkel}\label{exampleHT}

We now want to describe the stability chambers of the Fano variety of lines of certain smooth cubic fourfolds containing a smooth cubic scroll. We will see that in this case not all of the stability chambers are convex subcones of the big cone. In \cite{HassettTschinkel2010} the authors gave a complete description of the movable cone of such variety. The main result of Hassett and Tschinkel we are interested in is the following.

\begin{Thm}\cite[Theorem 7.4]{HassettTschinkel2010}\label{Thm:HassettTschinkel}
Suppose that $Y$ is a smooth cubic fourfold containing a smooth cubic
scroll $T$ with
$$H^4(Y,\mathbf{Z}) \cap H^{2,2}(Y,\mathbf{C})=\mathbf{Z} h^2 + \mathbf{Z} T$$
and let $F=F_0$ denote the variety of lines on $Y$.  
Then we have an infinite sequence of Mukai flops
$$\cdots F^{\vee}_2  \dashrightarrow F^{\vee}_1 \dashrightarrow F_0 \dashrightarrow
F_1 \dashrightarrow F_2 \cdots$$
with isomorphisms between every other flop in this sequence
\begin{equation}
\cdots F^{\vee}_{2} \cong F_0 \cong F_2  \cdots \quad \text{ and } \quad
\cdots F^{\vee}_{1} \cong F_1  \cdots  
\end{equation}
The positive cone of $F$ can be expressed as the union of the nef cones of the models $\dots, F^{\vee}_1, F_0, F_1, \dots$:

\begin{equation}
\cdots, \Cone(\alpha^{\vee}_2,\alpha^{\vee}_1), \Cone(\alpha^{\vee}_1,\alpha_1),\Cone(\alpha_1,\alpha_2),\cdots  
\end{equation} 
\end{Thm}
We refer the reader to \cite{HassettTschinkel2010} for all the geometric details of this example (which is quite intricate). We explain below the terminology of Theorem \ref{Thm:HassettTschinkel}. 
\vspace{0.2cm}

The birational geometry of $F$ is governed by the two planes $P, P^{\vee}$ and the surface $S$ contained in $F$ (cf.\ \cite[Proposition 4.1]{HassettTschinkel2010}). Let $f_1=f_P\colon F \DashedArrow F_1$ be the Mukai flop of $F$ along $P$, and $f_2=f_{P\cup S}\colon F \DashedArrow F_{2}$ the birational map obtained by flopping the plane $P\subset F$, first, and after $S_1\subset F_1$, where $S_1$ is the strict transform of the surface $S$ via $f_1$ (we can perform this Mukai flop because $S_1$ is a plane in $F_1$). One has $F_{2}\cong F$ (cf.\ \cite[Theorem 6.2]{HassettTschinkel2010}), so that $f_2$ gives a birational self map of order 2, i.e.\ an involution. Analogously we define
$f_1^{\vee} \colon F \DashedArrow F_1^{\vee}$ to be the Mukai flop of $F$ along $P^{\vee}$, and $f_2^{\vee}\colon F \DashedArrow F_2^{\vee}$ the birational map obtained by flopping the plane $P^{\vee}\subset F$, first, and after $S_1^{\vee}\subset F_1^{\vee}$, where $S_1^{\vee}$ is the strict transform of the surface $S$ via $f_1^{\vee}$ (we can perform this flop too because $S_1^{\vee}$ is a plane in $F_1^{\vee}$). Also, $f_1^{\vee}$ gives a birational involution. The composition $f_2\circ f_2^{\vee}$ is an element of $\mathrm{Bir}(F)$ of infinite order. The group $\mathrm{Bir}(F)$ is generated by $f_2^{\vee}$ and $f_2$, and any log MMP for pairs of the type $(F, D)$, with $[D]\in \mathrm{Mov}(F)$ starts by flopping first either $P$ or $P^{\vee}$, because the two extremal rays of the Mori cone of $F$ are generated by the classes of a line in $P$ and $P^{\vee}$.  To conclude, we define the birational map $f_3 \colon F \DashedArrow F_3$ as the composition of $f_2$ with the Mukai flop of $F_2$ along the strict transform $P^{\vee}_2$ via $f_2$ of $P^{\vee}$ (which is a plane), and the birational map $f_3^{\vee} \colon F \DashedArrow F_3^{\vee}$ as the composition of $f_2^{\vee}$ with the Mukai flop of $F_2^{\vee}$ along the strict transform $P_2$ via $f_2^{\vee}$ of $P$ (which also, in this case, is a plane).
\vspace{0.2cm}

Now, let $\tau$ be $\alpha([T])$,  where  $\alpha \colon H^4(Y,\mathbf{Z})\to H^2(F,\mathbf{Z})$ is the Abel-Jacobi map, $[T]$ is the class of the cubic scroll $T$ contained in $X$, and $g=\alpha(h^2)$, where $h^2$ is the square of the hyperplane class of the cubic fourfold. The pseudo-effective cone of $F$ is $\overline{\mathrm{Eff}(F)}=\mathrm{cone}\left(g-\left(3-\sqrt6\right)\tau,\left(3+\sqrt6\right)\tau-g\right)$. Note that we do not have prime exceptional divisors on $F$, hence $\mathrm{Big}(F)=\mathscr{C}_F\cap N^1(F)_{\mathbf{R}}=\mathrm{Mov}(F)=\mathrm{Eff}(F)$. One has $q_F(g)=q_F(g,\tau)=6$, and $q_F(\tau)=2$. Consider the following table of classes in $N^1(F)_{\mathbf{R}}$.
\vspace{0.2cm}

\begin{center}
\begin{tabular}{l|l|l}\label{table:classes}
 $\alpha_3^{\vee}=39\tau-7g$ & $\rho_3^{\vee}=8\tau-3/2\cdot g$ &  $f_{2,*}^{\vee}(\rho_3^{\vee})=\text{dual of the class of a line in } P_2^{\vee}$ \\ 
$\alpha_2^{\vee}=9\tau-g$ & $\rho_2^{\vee}=2\tau-g/2$  &  $f_{1,*}^{\vee}(\rho_2^{\vee})=\text{ dual of the class of a line in } S_1^{\vee}$   \\  
 $\alpha_1^{\vee}=g+3\tau$ & $\rho_1^{\vee}=\tau-g/2$  &  $\rho_1^{\vee}=$\text{ dual of the class of a line in } $P^{\vee}$\\  
 $\alpha_1=7g-3\tau$ & $\rho_1=3/2\cdot g-\tau$ & $\rho_1=$\text{ dual of the class of line in  } $P$ \\
 $\alpha_2=17g-9\tau$ & $\rho_2=7/2\cdot g-2\tau$ &  $f_{1,*}(\rho_2)=\text{ dual of the class of a line in } S_1$\\
 $\alpha_3=71g-39\tau$ & $\rho_3=29/2\cdot g- 8 \tau$ &  $f_{2,*}(\rho_3)=\text{ dual of the class of a line in } P_2$
\end{tabular}
\end{center}
\vspace{0.2cm}

 With the above notation, we obtain the following decomposition of the big cone into stability chambers.
\begin{Prop}
 We have for $\mathrm{Big}(F)$ six stability chambers. The first is  $\mathrm{Amp}(F)$, the ample cone of $F$. In the second, $\mathrm{SC}_1$, the augmented base locus consists of the plane $P$. In the third, $\mathrm{SC}_1^{\vee}$, the augmented base locus is given by the plane $P^{\vee}$. In the fourth, $\mathrm{SC}_2$, the augmented base locus is given by $P\cup S$, and in the fifth, $\mathrm{SC}_2^{\vee}$, the augmented base locus is given by the union $P^{\vee} \cup S$. The sixth stability chamber, which is not a convex cone, is denoted by $\mathrm{SC}_3$, and there the augmented base locus is given by the union $P\cup P^{\vee} \cup S$. 
 \begin{proof}
    The ample cone $\mathrm{Amp}(F)$ is the stability chamber where the augmented base locus is empty, which we denote by $\mathrm{SC}_0$. A log MMP for any divisor lying in $f_P^{*}(\mathrm{Amp}(F_1))$ is given by $f_P$. Then, the augmented base locus of any class lying in $f_P^{*}(\mathrm{Amp}(F_1))$ is given by the plane $P$. Note that the rays $\alpha_1$ and $\alpha_2$ are part of the unstable locus (i.e.\ the regions of the big cone where $\mathbf{B}_{-} \subsetneq \mathbf{B}_+$). Then, $\mathrm{SC}_1=f_P^{*}(\mathrm{Amp}(F_1))\cup \alpha_1$ gives a second stability chamber. Similarly, the augmented base locus of any class lying in $f_P^{*}(\mathrm{Amp}(F_1^{\vee}))$ is given by $P^{\vee}$, and so $\mathrm{SC}_1^{\vee}=f_P^{*}(\mathrm{Amp}(F_1^{\vee}))\cup \alpha_1^{\vee}$ gives the third stability chamber. The augmented base locus of any class lying in $f_{P\cup S}^{*}(\mathrm{Amp}( F_2))$ is given by $P\cup S$, because $f_{P\cup S}$ gives a log MMP for any such class. It follows that $\mathrm{SC}_{2}=f_{P\cup S}^{*}(\mathrm{Amp}( F_2)) \cup \alpha_2$ is a stability chamber. Analogously, we see that the augmented base locus of any class lying in $f_{P^{\vee}\cup S}^{*}(\mathrm{Amp}( F_2^{\vee}))$ is given by $P^{\vee}\cup S$, and hence $\mathrm{SC}_2^{\vee}=f_{P^{\vee}\cup S}^{*}(\mathrm{Amp}( F_2^{\vee}))\cup \alpha_2^{\vee}$ is another stability chamber. To conclude, we see that the augmented base locus of any class lying in $f^{*}_{P\cup S \cup P^{\vee}}(\mathrm{Amp}(F_3))$ is given by $P\cup S \cup P^{\vee}$, because $f_{P\cup S \cup P^{\vee}}$ gives a log MMP for these classes. But the same is true for any element lying in $f^{*}_{P^{\vee}\cup S \cup P}(\mathrm{Amp}(F_3^{\vee}))$. After this step, the augmented base locus stabilizes "on both sides" of the big cone (i.e.\ on the left of the ray $\alpha_3^{\vee}$, and on the right of $\alpha_3$), and it is always equal to $P^{\vee}\cup S\cup P$. The sixth chamber, denoted by $\mathrm{SC}_3$, is given by $\mathrm{Big}(F)$ minus the interior of the convex cone spanned by $\alpha_3$ and $\alpha_3^{\vee}$. In particular, we observe that it is not a convex subcone of $\mathrm{Big}(F)$. This is consistent with Lemma \ref{lem:convexity}. Indeed, for example, if $\gamma$ is any class lying in the interior of $\mathrm{SC_3}$ and $l^{\vee}$ (resp.\ $l$) is any line lying in $P^{\vee}$ (resp.\ $P$), we see that either $\gamma \cdot l^{\vee}>0$, or\ $\gamma \cdot l>0$. But $l^{\vee} \subset \mathbf{B}_+(\gamma)=\mathbf{B}_{-}(\gamma)$ (resp.\ $l \subset \mathbf{B}_+(\gamma)=\mathbf{B}_{-}(\gamma)$). In particular, to characterize numerically the asymptotic base loci, we cannot avoid considering the birational models of the variety. To conclude, we observe that unstable locus of $\mathrm{Big}(F)$ is given by the union of rays $\alpha_1\cup \alpha_1^{\vee}\cup \alpha_2 \cup \alpha_2^{\vee} \cup \alpha_3 \cup \alpha_3^{\vee}$.
 \end{proof}
\end{Prop}
\begin{Rem}
 Let $F$ be as in Theorem \ref{Thm:HassettTschinkel}. According to Theorem \ref{Thm:duality}, we provide an example of a class in $N_1(F)_{\mathbf{R}}$ which on $F_2$ is represented by a rational curve moving in a family sweeping out a subvariety of $F_2$, which is the strict transform of a $2$-dimensional subvariety of $F$. Following the notation of the table above, consider the class $\rho_3\in N_1(F)_{\mathbf{R}}$, and $f_{2,*}(\rho_3) \in  N_1(F_2)_{\mathbf{R}}$. Then, the class dual to $f_{2,*}(\rho_3)$ is represented by a line moving in a family sweeping out the lagrangian plane $f_{2,*}(P^{\vee})$ (the strict transform of $P^{\vee}$ via $f_2$). The extremal rays of the Mori cone of $F$ are generated by $q_F(\rho_1,-)$, and $q_F(\rho^{\vee}_1,-)$, and an easy calculation shows $q_F(\rho_3,-)=\frac{5}{2}q_F(\rho_1^{\vee},-)+\frac{21}{2}q_F(\rho_1,-)$.
 
\end{Rem}
\begin{figure}[htbp]\label{picture}
  \centering

\definecolor{light-gray}{gray}{0.95}
  \begin{tikzpicture}[scale=0.7, x=1.6
cm,y=0.8cm]

       \node (l) at (-2,10.4) {$\alpha_1^\vee$};

       \node (l) at (2,10.3) {$\alpha_1$};

       \node (l) at (-3.7,10.4) {$\alpha_2^\vee$};

       \node (l) at (3.7,10.3) {$\alpha_2$};

       \node (l) at (-4.5,10.4) {$\alpha_3^\vee$};

       \node (l) at (4.5,10.3) {$\alpha_3$};

\draw [line width=0.5pt] (0,0) -- (5,10);

\draw [line width=0.5pt] (0,0) -- (-5,10);

\draw [MistyRose1, line width=0.5pt](0,0) -- (-2,10);

\draw [LightCyan1, line width=0.5pt] (0,0) -- (2,10);

  \path[fill=LightCyan1] (0,0) -- (-2,10) -- (2,10) -- cycle;
  \path [fill=MistyRose1](0,0) -- (-2,10);

  \path[fill=DarkOliveGreen1] (0,0) -- (2,10) -- (3.7,10) -- cycle;

    \path[fill=MistyRose1] (0,0) -- (-2,10) -- (-3.7,10) -- cycle;

   \path[fill=LemonChiffon1] (0,0) -- (-3.7,10) -- (-4.5,10) -- cycle;

          \path[fill=Cornsilk1] (0,0) -- (3.7,10) -- (4.5,10) -- cycle;

          \path[fill=Snow3] (0,0) -- (-4.5,10) -- (-5,10) -- cycle;

             \path[fill=Snow3] (0,0) -- (4.5,10) -- (5,10) -- cycle;

\draw [LemonChiffon1, line width=0.5pt] (0,0) -- (-3.7,10);
\draw [AntiqueWhite1, line width=0.5pt] (0,0) -- (3.7,10);

\draw [Snow3, line width=0.5pt] (0,0) -- (-4.5,10);
\draw [Snow3, line width=0.5pt] (0,0) -- (4.5,10);

       \node (sigma0) at (0,6) {\small $\mathrm{SC}_0$};
        \node (sigma1) at (2.4,8) {\small $\mathrm{SC}_1$};
          \node (sigma1dual) at (-2.4,8) {\small $\mathrm{SC}_1^{\vee}$};

          \node (pt1dual) at (-2.5,3){\small $\mathrm{SC}_2^{\vee}$};
           \node (pt1) at (2.5,3){\small $\mathrm{SC}_2$};

           \node (pt) at (2,-1){\small $\mathrm{SC}_3$};

           \draw [->] (pt1dual) to[bend left] (-3.3,8);

           \draw [->] (pt1) to[bend right] (3.3,8);

              \draw [->] (pt) to[out=-150,in=270] (-4.5,9.5);

           \draw [->] (pt) to[bend right] (4.5,9.5);

       \fill(0,0) circle (2 pt);
       \node (o) at (-0.3,-0.3) {$0$};
  \end{tikzpicture}
  \caption{Representation of the stability chambers on $F$}
  \label{fig:cones}
\end{figure}

\begin{Rem}\label{Morifinerthanstabchambers}
Notice that in this example the decomposition of $\mathrm{Eff}(F)=\mathrm{Big}(F)$ into Mori chambers is strictly finer than the one into stability chambers. Indeed, while any Mori chamber is equal to $f^{-1}(\mathrm{Nef}(F'))$ for some birational model $F'$ of $F$ so that we have infinitely many ones, the stability chambers are six.
\end{Rem}

We would like to conclude this section with a result which is interesting on its own.
\vspace{0.2cm}

If $X$ is a projective HK manifold carrying a birational self-map of infinite order, $\mathrm{Bir}(X)$ is an infinite group. This happens for instance in the above example. Indeed, as it is proven in \cite{HassettTschinkel2010}, with the same notation of above, the map $f_{P\cup S}$ gives an involution $\iota \in \mathrm{Bir}(X)$, and $f_{P^{\vee} \cup S}$ induces a second involution $\iota^{\vee} \in \mathrm{Bir}(X)$. An element of infinite order is then $\iota \circ \iota^{\vee}$ (but also $\iota^{\vee}\circ \iota$). 
We below show that also the converse is true, using a classical result of Burnside of representation theory, and a result of Cattaneo-Fu.

\begin{Thm}[\cite{Burn}, Main Theorem]\label{burnthm}
Let $G$ be a subgroup of $\mathrm{GL}(n,\mathbf{C})$. Assume that there exists a positive integer $d$
such that any element of $G$ has order at most $d$. Then $G$ is a finite group.
\end{Thm}

\begin{Prop}
Suppose that $\mathrm{Bir}(X)$ is infinite. Then $\mathrm{Bir}(X)$ contains an element of infinite order.
\begin{proof}
By \cite[Theorem 7.1]{CattaneoFu2019} the number of conjugacy classes of finite subgroups of $\mathrm{Bir}(X)$ is finite. Then, there exists an integer $N$ such that any element of finite order in $\mathrm{Bir}(X)$ has at most order $N$. Suppose by contradiction that $\mathrm{Bir}(X)$ does not contain elements of infinite order. Then the image of the natural representation
\[
\rho \colon \mathrm{Bir(X)} \to \mathrm{GL}(N^1(X)_{\mathbf{C}})
\]
is a group of finite order, by Theorem \ref{burnthm}. But then $\mathrm{Bir}(X)$ is finite by \cite[Corollary 2.6]{Ogu14}, and this is a contradiction.
\end{proof}
\end{Prop}

\subsection{\large \itshape A non-equidimensional augmented base locus}

Let $S$ be a projective K3 surface carrying a smooth rational curve $C \cong \mathbf{P}^1$. Consider the $n$-plane $\mathbf{P}^n\cong C^{[n]}\subset  S^{[n]} $. Let $\mathrm{cont} \colon S\to S'$ be the morphism contracting the curve $C$. We have an induced morphism $\mathrm{cont}^{(n)} \colon S^{(n)}\to S'^{(n)}$, contracting $C^{(n)}$. The $n$-plane $C^{[n]}$ is not contracted by the Hilbert-Chow morphism $\pi\colon S^{[n]}\to S^{(n)}$, hence $C^{[n]} \not\subset E$, where $E$ is the exceptional divisor coming from $\pi$. Let $D$ be the wall divisor related to $C^{[n]}$. Then, if $A$ is any ample, integral, Cartier divisor on $S'^{(n)}$, we have $D':=\left(\pi\circ \mathrm{cont}^{(n)}\right)^{*}(A) \in  D^{\perp} \cap E^{\perp}$, and $E$, $C^{[n]}$ are two irreducible components of $\mathbf{B}_+(D')$ of different dimension.

\subsection{\large \itshape Flopping contractions after Bayer-Macrì}

All the examples discussed above can be obtained as moduli spaces of Bridgeland stable sheaves on a K3 surface $S$. The birational geometry of these spaces is governed by a chamber and wall decomposition on a connected component $\Stab^\dagger(S)$ of the space of stability conditions on $\mathrm{D}^b(S)$. We have shown that the decomposition in stability chambers can be read in terms of the contraction loci, hence the decomposition of the big cone into stability chambers is encoded in terms of wall-crossing in $\Stab^\dagger(S)$. To mention one consequence of such decomposition we refer to \cite[Examples 14.3-14.4, Remark 14.5]{BayerMacri2014} where examples of contraction loci (and hence augmented base loci) with several irreducible components or arbitrarily many connected components are given.

\section{\bf \scshape Final Remarks and Some Questions}\label{Section8}
In this section, we ask two questions naturally arising from our work.

\subsection{\large \itshape Is the restricted base locus of divisors on HK manifolds Zariski-closed?}

In general, given a pseudo-effective divisor $D$ on a smooth projective variety $Y$, the restricted base locus $\mathbf{B}_-(D)$ is not Zariski closed. The first examples of non-Zariski closed restricted base loci were provided by John Leusieutre (cf.\ \cite{lesieutre}), and we do not know other examples of this pathology.
\vspace{0.2cm}

In the case of a projective HK manifold $X$, if  $\mathrm{Bir}(X)$ is finite, $\mathbf{B}_-(D)$ is Zariski closed for any pseudo-effective divisor $D$. In any case, if $D$ is big on $X$, as $\mathbf{B}_-(D)=\mathbf{B}(D)$, all big divisors have Zariski-closed restricted base locus. Hence, the only classes that might have a non-Zariski closed restricted base locus on HK manifolds, are the pseudo-effective but not big ones. In the example of Hassett and Tshinkel we observe that all pseudo-effective classes have a Zariski closed restricted base locus. Indeed, in that case, the only potential rays with a non-Zariski closed restricted base locus are the ones spanned by $\alpha=g-(3-\sqrt6)\tau$ and $\beta=(3+\sqrt6)\tau-g$. Using that $\mathbf{B}_-(D)=\cup_{m \in \mathbf{N}} \mathbf{B}_+(D+A_m)$, where $\{A_m\}_{m \in \mathbf{N}}$ is a sequence of ample divisors with $\lVert A_m \rVert \to 0$, and that $\mathbf{B}_+(\alpha+A_m)$ and $\mathbf{B}_+(\beta+A_m)$ stabilize for $m\gg 0$, we see that $\mathbf{B}_-(\beta)=\mathbf{B}_-(\alpha)=P\cup P^{\vee} \cup S$, so that all pseudo-effective divisors on $F$ have a Zariski-closed restricted base locus. Then, it is natural to ask the following question, which will be investigated by the authors in future works.

\begin{Quesone*}
Let $X$ be a projective HK manifold and $D$ any pseudo-effective divisor on $X$. Is $\mathbf{B}_-(D)$ Zariski closed?
\end{Quesone*}

\subsection{\large \itshape Is the Voisin filtration preserved under birational maps?}
Our results shed light on Conjecture \cite[Conjecture 0.4]{Voisin16} and also motivate the following question about the birational invariance of the Voisin filtration.
\begin{Questwo*}\label{Con:birationalinvarianceofVoisinfiltration}
Let $f: X\DashedArrow X'$ be a birational map between projective HK manifolds. Is it true that the induced isomorphism
\begin{equation}
    \mathrm{CH}_0(X)\cong \mathrm{CH}_0(X')
\end{equation}
preserves the filtration $S_\bullet$?
\end{Questwo*}
In the case of moduli space of Bridgeland stable sheaves on a K3 surface, the above question was answered affirmatively very recently in (cf.\ \cite[Corollary 3.3]{LiZhang2023}).

\section*{\bf \scshape Acknowledgments}
The first-named author would like to thank his doctoral advisors, Giovanni Mongardi and Gianluca Pacienza for their constant support. We would like to thank Daniele Agostini, Enrico Fatighenti, Enrica Floris, Annalisa Grossi, Christian Lehn, Claudio Onorati, Andrea Petracci, and Jieao Song for useful discussions. To conclude, we would like to thank the anonymous referee for his suggestions, which helped us to improve the exposition of the article. Part of this work was done during a visit of the first-named author at Sapienza Università di Roma. The visit was funded with the research grant “SEED PNR” of Simone Diverio, which we would like to thank. The second named author was partially supported by the SMWK research grant SAXAG. Both the authors were partially supported by the ERC Synergy Grant HyperK (Grant agreement ID: 854361).

\printbibliography

\end{document}